\documentclass[11pt]{amsart}
\usepackage{amsmath,amssymb}
\usepackage[dvips]{graphicx}

\usepackage{fullpage}

\def\COMMENT#1{}
\let\COMMENT=\footnote
\def\TASK#1{}

\def\noproof{{\unskip\nobreak\hfill\penalty50\hskip2em\hbox{}\nobreak\hfill%
        $\square$\parfillskip=0pt\finalhyphendemerits=0\par}\goodbreak}
\def\endproof{\noproof\bigskip}
\newdimen\margin   
\def\textno#1&#2\par{%
    \margin=\hsize
    \advance\margin by -4\parindent
           \setbox1=\hbox{\sl#1}%
    \ifdim\wd1 < \margin
       $$\box1\eqno#2$$%
    \else
       \bigbreak
       \hbox to \hsize{\indent$\vcenter{\advance\hsize by -3\parindent
       \sl\noindent#1}\hfil#2$}%
       \bigbreak
    \fi}
\def\proof{\removelastskip\penalty55\medskip\noindent{\bf Proof. }}

\def\eps{\varepsilon}

\def\a{\alpha}
\def\b{\beta}
\def\d{\delta}
\def\g{\gamma}
\def\u{\underline}
\def\B{\mathcal{B}}
\def\oB{\overline{\mathcal B}}

\def\dB{\mathcal B _{n,k}}
\def\doB{\overline{\mathcal B} _{n,k}}

\newcounter{lth}
\newcounter{2th}
\thanks {$^\fnsymbol{lth}$This research was undertaken whilst the first author was a researcher at Charles University, Prague. The work leading to this invention has received funding from the European Research Council under the European Union's Seventh Framework Programme (FP7/2007-2013)/ERC grant agreement no.~259385.}
\thanks{$^\fnsymbol{2th}$The second author was partially supported by NSA grants H98230-10-1-0165 and
H98230-12-1-0283.}

\newtheorem{firstthm}{Proposition}[section]
\newtheorem{thm}[firstthm]{Theorem}
\newtheorem{prop}[firstthm]{Proposition}
\newtheorem{lemma}[firstthm]{Lemma}

\newtheorem{claim}[firstthm]{Claim}

\begin{document}
\title{Exact minimum degree thresholds for perfect matchings in uniform hypergraphs II}
\author{Andrew Treglown$^\fnsymbol{lth}$ and Yi Zhao$^\fnsymbol{2th}$}
\date{\today}

\begin{abstract}
Given positive integers $k\geq 3$ and $\ell$ where  $k/2 \leq \ell \leq k-1$,
we give a  minimum $\ell$-degree condition that ensures a perfect matching in a $k$-uniform hypergraph.
This condition is best possible and improves on work of
Pikhurko~\cite{pik} who gave an asymptotically exact result. Our approach makes use of the absorbing method, and builds
on work in~\cite{zhao}, where we proved the result for $k$ divisible by $4$.
\end{abstract}

\maketitle

\section{Introduction}\label{sec1}
A central question in graph theory is to establish conditions that ensure a (hyper)graph $H$ contains some spanning (hyper)graph $F$. 
Of course, it is desirable to fully characterize those (hyper)graphs $H$ that contain a  spanning copy of a given
(hyper)graph $F$. 
Tutte's theorem~\cite{tutte} characterizes those graphs with a perfect matching.
(A \emph{perfect matching} in a (hyper)graph $H$ is a collection of vertex-disjoint edges of $H$ which cover the vertex set $V(H)$ of $H$.) However, for some (hyper)graphs $F$  it is unlikely that
such a characterization exists. 
Indeed, for many (hyper)graphs $F$  the decision problem of whether a (hyper)graph $H$ contains $F$ is NP-complete.
For example, in contrast to the graph case, the decision problem whether a $k$-uniform hypergraph contains a perfect matching is NP-complete for $k\geq 3$ (see~\cite{karp, garey}).
Thus, it is desirable to find sufficient conditions that ensure a perfect matching in a $k$-uniform hypergraph.

Given a $k$-uniform hypergraph $H$ with an $\ell$-element vertex set $S$ (where $0 \leq \ell \leq k-1$) we define
$d_H (S)$ to be the number of edges containing $S$. The \emph{minimum $\ell$-degree $\delta _{\ell}
(H)$} of $H$ is the minimum of $d_H (S)$ over all $\ell$-element sets $S$ of vertices in $H$. Clearly $\delta_0(H)$ is the number of edges in $H$. We also
refer to  $\delta _1 (H)$ as the \emph{minimum vertex degree} of $H$ and  $\delta _{k-1}
(H)$ the \emph{minimum codegree} of $H$.

Over the last few years there has been a strong focus in establishing minimum $\ell$-degree thresholds  that force a perfect matching in
a $k$-uniform hypergraph. See~\cite{rrsurvey} for a survey on matchings (and Hamilton cycles) in hypergraphs. In particular,
R\"odl, Ruci\'nski and Szemer\'edi~\cite{rrs} determined the minimum codegree threshold that ensures  a perfect matching in
a $k$-uniform hypergraph on $n$ vertices for all $k\ge 3$. The threshold is $n/2 - k + C$, where $C\in \{3/2, 2, 5/2, 3\}$ depends on the values of $n$ and $k$. This improved bounds given in~\cite{ko1, rrs1}.

Less is known about minimum vertex degree thresholds  that force a perfect matching.
One of the earliest results on perfect matchings was given by Daykin and H\"aggkvist \cite{dayhag}, who showed that
a $k$-uniform hypergraph $H$ on $n$ vertices contains a perfect matching provided that $\delta_1(H)\ge (1 - 1/k)\binom{n-1}{k-1}$.
H\`an, Person and Schacht~\cite{hps} determined, asymptotically, the minimum vertex degree that forces a perfect matching
in a $3$-uniform hypergraph.
K\"uhn, Osthus and Treglown~\cite{KOTmatch} and independently
Khan~\cite{khan1} made this result exact.
Khan~\cite{khan2} has also determined the exact minimum vertex degree threshold for $4$-uniform
hypergraphs.
For $k\geq 5$, the precise minimum vertex degree threshold which ensures a perfect matching in a $k$-uniform hypergraph is not
known.

The situation for $\ell$-degrees where $1 < \ell < k-1$ is also still open.
H\`an, Person and Schacht~\cite{hps} provided conditions on $\delta _{\ell} (H)$ that ensure a perfect matching in the case
when $1 \leq \ell< k/2$. These bounds were subsequently lowered by Markstr\"om and Ruci\'nski~\cite{mark}.
 Alon et al.~\cite{afh} gave a connection between the minimum $\ell$-degree
that forces a perfect matching in a $k$-uniform hypergraph and the  minimum $\ell$-degree
that forces a \emph{perfect fractional matching}. As a consequence of this result they determined, asymptotically,
the minimum $\ell$-degree which forces a perfect matching in a $k$-uniform hypergraph for the following values of
$(k, \ell)$: $(4, 1)$, $(5, 1)$, $(5, 2)$, $(6, 2)$, and  $(7, 3)$.

Pikhurko~\cite{pik} showed that if $\ell \geq k/2$ and $H$ is a $k$-uniform hypergraph whose order $n$ is divisible by $k$ then
$H$ has a perfect matching provided that $\delta _{\ell} (H) \geq (1/2+o(1))\binom{n}{k-\ell}$. This result is best possible up to the $o(1)$-term
(see the constructions in $\mathcal H_{\text{ext}} (n,k)$ below). 

In this paper we make Pikhurko's result exact. In order to state our main result, we need some more definitions. Fix a set $V$ of $n$ vertices. Given a partition of $V$ into non-empty sets $A, B$,
let $E^k _{\text{odd}}(A, B)$ ($E^k _{\text{even}}(A, B)$) denote the family of all $k$-element subsets of $V$ that intersect $A$ in an odd (even) number of vertices.
(Notice that the ordering of the vertex classes $A,B$ is important.)
When it is clear from the context, we write, for example, $E _{\text{odd}}(A, B)$.
Define $\mathcal B_{n,k}(A,B)$ to be the $k$-uniform hypergraph with vertex set $V= A \cup B$ and edge set $E_{\text{odd}} (A,B)$.
Note that the complement $\overline{\mathcal B}_{n,k} (A,B)$ of $\mathcal B_{n,k} (A,B)$ has edge set $E_{\text{even}} (A,B)$.

Suppose $n,k \in \mathbb N$ such that $k$ divides $n$. Define $\mathcal H_{\text{ext}} (n,k)$ to be the collection of the following hypergraphs: $\mathcal H_{\text{ext}} (n,k)$ contains all hypergraphs $\doB (A,B)$ where $|A|$ is odd. Further, if $n/k$ is odd 
then $\mathcal H_{\text{ext}} (n,k)$ also contains all hypergraphs $\dB (A,B)$ where $|A|$ is even; if $n/k$ is even then $\mathcal H_{\text{ext}} (n,k)$ also contains all hypergraphs $\dB (A,B)$ where $|A|$ is odd.

It is easy to see that no hypergraph in $\mathcal H_{\text{ext}} (n,k)$ contains a perfect matching. Indeed, first assume that $|A|$ is even and $n/k$ is odd. Since every edge of ${\mathcal B}_{n,k}(A, B)$ intersects $A$ in an odd number of vertices, one cannot cover $A$ with an odd number of disjoint odd sets. Similarly $\mathcal B_{n,k} (A, B)$ does not
contain a perfect matching if $|A|$ is odd and $n/k$ is even. Finally, if $|A|$ is odd then since every edge of $\overline{\mathcal B}_{n,k}(A,B)$ intersects $A$ in an even number of vertices, $\overline{\mathcal B}_{n,k}(A, B)$ does not contain a perfect matching.

Given $\ell \in \mathbb N$ such that $k/2 \leq  \ell \leq k-1$ define $\delta (n,k, \ell)$ to be the maximum of the minimum $\ell$-degrees among all the hypergraphs in  $\mathcal H_{\text{ext}} (n,k)$. For example, it is not hard to see that
\begin{equation} \label{eq:nk1}
\delta(n, k, k-1) = \left\{\begin{array}{ll}
{n}/{2} - k + 2 & \text{if $k/2$ is even and $n/k$ is odd}\\
{n}/{2} - k + {3}/{2} & \text{if $k$ is odd and $(n-1)/{2}$ is odd}\\
{n}/{2} - k + {1}/{2} & \text{if $k$ is odd and $(n-1)/{2}$ is even} \\
{n}/{2} - k + 1 & \text{otherwise.}
\end{array} \right.\end{equation}

The following is our main result.
\begin{thm}\label{mainthm}
Let $k, \ell \in \mathbb N$ such that $k\ge 3$ and $k/2\leq \ell \leq k-1$. Then
there exists an $n_0 \in \mathbb N$ such that the following holds. Suppose $H$ is a $k$-uniform hypergraph on
$n \geq n_0$ vertices where $k$ divides $n$. If
$$\delta _{\ell} (H) > \delta (n,k,\ell)$$
then $H$ contains a perfect matching.
\end{thm}
In~\cite{zhao}, we proved Theorem~\ref{mainthm} in the case when $4$ divides $k$. Independently to this,
Czygrinow and Kamat~\cite{czy} proved Theorem~\ref{mainthm} in the case when $k=4$ and $\ell =2$.
To prove Theorem~\ref{mainthm}
we use several ideas and results from~\cite{zhao}. In particular, the so-called `extremal' case of Theorem~\ref{mainthm}
was proved in~\cite{zhao} for \emph{all} values of $k$.
However, in some parts of the proof of the `non-extremal' case we use a very different
approach to that in~\cite{zhao}. We discuss this in more detail in Section~\ref{presec}.

As explained before, the minimum $\ell$-degree condition in Theorem~\ref{mainthm} is best possible.
Theorem~\ref{mainthm} and \eqref{eq:nk1} together give the aforementioned result of R\"odl, Ruci\'nski and Szemer\'edi \cite{rrs}.

In general, the precise value of $\delta (n,k, \ell)$ is unknown because it is not known what value of $|A|$ maximizes  the minimum $\ell$-degree of $\mathcal B _{n,k} (A, B)$ (or $\overline{\mathcal B} _{n,k} (A, B)$). (See~\cite{zhao} for a discussion on this.)
However, in~\cite{zhao} we gave a tight upper bound on $\delta (n,4,2)$.

\section{Notation and preliminaries}
\label{sec2}
\subsection{Definitions and notation}
Given a set $X$ and  $r\in \mathbb N$, we write $\binom{X}{r}$ for the set of all $r$-element subsets ($r$-subsets, for short) of $X$. 
Given a set $S$ and an element $x$, we often write $S- \{x\}$ as $S-x$ and $S\cup \{x\}$ as $S+x$. Let $k \in \mathbb N$. A $k$-uniform hypergraph
$H$ consists of a set of \emph{vertices} $V(H)$ and a set of \emph{edges} $E(H) \subseteq \binom{V(H)}{k}$. So in the case when
$k=1$, $E(H) \subseteq V(H)$. (The notion of a $1$-uniform hypergraph will be used in Section~\ref{sectionnon}.)

Let $k, \ell \in \mathbb N$. Suppose $H=(V(H),E(H))$ is a $k$-uniform hypergraph.  
Let $\{v_1, \dots, v_{\ell} \}$ be an $\ell$-subset of $V(H)$. Often we write it as $v_1\dots v_{\ell}$ (i.e. we drop the brackets and commas), or simply $\u{v}$.
Given $\underline{v}\in \binom{V(H)}{\ell}$,  we write $N_H(\underline{v})$ or $N (\underline{v})$ to denote the \emph{neighborhood of $\u{v}$}, that is, the family of those $(k-\ell)$-subsets of $V(H)$ which, together with $\u{v}$, form an edge in $H$. Then $|N_H(\underline{v})| = d_H(\underline{v})$. Given a vertex $v \in V(H)$, we define $N_H (v)$ and $d_H (v)$ analogously.

We denote the \emph{complement of $H$} by $\overline{H}$. That is, $\overline{H} := (V(H), \binom{V(H)}{k}\setminus E(H))$.  Given a set $A \subseteq V(H)$, $H[A]$ denotes the $k$-uniform subhypergraph of $H$ \emph{induced by $A$}, namely, $H[A] := (A, E(H)\cap \binom{A}{k})$. 
Given $B \subseteq E(H)$, we define $H[B] := (V(H), B)$.

Let $A,B$ be sets and let $m$ be a positive real. Let $A\triangle B:= (A\setminus B) \cup (B\setminus A)$ denote the \emph{symmetric difference} of  $A$ and $B$. We write $A=B \pm m$ if $|A\triangle B|\leq m$.

Let $\eps>0$. Suppose that $H$ and $H'$ are $k$-uniform hypegraphs on $n$ vertices. We say that $H$ is \emph{$\eps$-close to $H'$}, and write $H= H'\pm \eps n ^{k}$, if $H$ becomes a copy of $H'$ after adding and deleting at most $\eps n^k$ edges. More precisely,  $H$ is $\eps$-close to $H'$ if there is an isomorphic copy $\tilde{H}$ of $H$ such that $V(\tilde{H}) = V(H')$ and $|E(\tilde{H})\triangle E(H')| \le \eps n^k$.

Given a graph $G$, $x \in V(G)$ and $Y \subseteq V(G)$, we denote by $d_G (x,Y)$ the number of vertices $y \in Y$ such that
$xy \in E(G)$. Given disjoint $A,B \subseteq V(G)$ we let $e(A,B)$ denote the number of edges in $G$ with one endpoint in 
$A$ and one endpoint in $B$. Further, we let $K_{A,B}$ denote the complete bipartite graph with vertex classes $A$ and $B$.

We will often write $0<a_1 \ll a_2 \ll a_3$ to mean that we can choose the constants
$a_1,a_2,a_3$ from right to left. More
precisely, there are increasing functions $f$ and $g$ such that, given
$a_3$, whenever we choose some $a_2 \leq f(a_3)$ and $a_1 \leq g(a_2)$, all
calculations needed in our proof are valid.
Hierarchies with more constants are defined in the obvious way.
Throughout the paper we omit floors and ceilings whenever this does not affect the argument.

\subsection{The extremal graph $\mathcal B _{n,k}$ and absorbing sets.}
Suppose that $n,k\in \mathbb N$ such that $n\geq k $. Let $A,B$ be a partition of a set of $n$ vertices.
Recall that $\mathcal B_{n,k}(A,B)$ is the $k$-uniform hypergraph with vertex set $A \cup B$ and edge set $E_{\text{odd}} (A,B)$, and its complement $\overline{\mathcal B}_{n,k}(A, B)$ has edge set $E_{\text{even}} (A,B)$.
(Note that ${\mathcal B}_{n,1}(A, B)$ has edge set $A$ and $\overline{\mathcal B}_{n,1}(A, B)$ has edge set $B$.)
When $|A|=\lfloor n/2 \rfloor$ and $|B|=\lceil n/2 \rceil$, we simply denote $\mathcal B_{n,k}(A,B)$ by $\mathcal B_{n,k}$, and $\overline{\mathcal B}_{n,k}(A, B)$ by $\overline{\mathcal B}_{n,k}$.

Following the ideas of R\"odl, Ruci\'nski and Szemer\'edi \cite{rrs2, rrs}, we define \emph{absorbing sets} as follows:
Given a $k$-uniform hypergraph $H$,  a set $S\subseteq V(H)$ is called an \emph{absorbing set for $Q\subseteq V(H)$}, if both $H[S]$ and $H[S\cup Q]$ contain perfect matchings. In this case, if the matching covering $S$ is $M$, we also say \emph{$M$ absorbs $Q$}.

\subsection{Useful results}\label{usesec}
When considering $\ell$-degree together with $\ell'$-degree for some $\ell'\neq \ell$, the following proposition is very useful (the proof is a standard counting argument, which we omit).

\begin{prop} \label{prop:deg}
Let $0\le \ell \le \ell' < k$ and $H$ be a $k$-uniform hypergraph. If $\d_{\ell'}(H)\geq x\binom{n- \ell'}{k- \ell'}$ for some $0\le x\le 1$, then $\d_{\ell}(H)\geq x\binom{n- \ell}{k- \ell}$.
\end{prop}
The following two results are applied in Section~\ref{lastlemma}. Given an $r$-uniform hypergraph $F=(V,E)$ with two distinct vertices $u, v \in V$, define $D_F(u, v)$ as the family of $(r+1)$-subsets $S\in \binom{V}{r+1}$ such that $u,v \in S$ 
and either $S- u \in E$ and $S-v \not\in E$, or  $S- u\not\in E$ and $S- v \in E$. Note that $D_F(u, v)= D_F(v, u) = D_{\overline{F}}(u,v)$, where $\overline{F}$ is the complement of $F$.

\begin{lemma}\label{lem:KK}
Given any $r \in \mathbb N$ and $\alpha >0$ there exists an $n_0 \in \mathbb N$ such that the following holds.
Let $F=(V, E)$ be an $r$-uniform hypergraph on $n \geq n_0$ vertices with edge density $\rho:= |E|/\binom{n}{r}\in [\a, 1 - \a]$. 
Then
\begin{equation}\label{eq:Nv}
\sum_{\{v, v'\}\in \binom{V}{2}} |D_F(v, v')| \ge \a (1-\a) \binom{n}{r+1}.
\end{equation}
In particular, there exists two vertices $v, v'\in V$ such that $| D_F(v, v') | \ge \a (1- \a) n^{r-1}/(r+1)!$.
\end{lemma}
\begin{proof}
If $r=1$ then the second assertion is trivial.  If $r\geq 2$,
the second assertion follows by an averaging argument: if the first assertion holds, then there exist two vertices $v$ and $v'$ in $V$ such that 
\begin{align*}
| D_F(v, v') |  \ge \frac{\a(1-\a)}{\binom{n}{2}} \binom{n}{r+1}= \a(1-\a) \frac{2(n-2) \cdots (n-r)}{(r+1)!}  \ge \a(1-\a) \frac{n^{r-1}}{(r+1)!},
\end{align*}
since $n$ is sufficiently large.

We prove the first assertion by double counting and the Kruskal--Katona theorem. Let $m$ denote the  left-hand side of \eqref{eq:Nv}, that is, the number of $(r+1)$-subsets $S\subset V$ that contains two labeled vertices $v\ne v'$ such that exactly one of $S - v$ and $S- v'$ is in $E$.
For $0\le i\le r+1$, let $t_i$ denote the number of $(r+1)$-subsets of $V$ that span exactly $i$ edges of $F$. It is easy to see that
\begin{align*}
 |E| (n-r)  = \sum_{i=1}^{r+1} i t_i \  \ \text{and} \ \ m  = \sum_{i=1}^r  i (r+1 - i) t_i.
\end{align*}
Thus,
\begin{equation}\label{eq:m}
m\ge \sum_{i=1}^r i t_i = |E| (n -r) - (r+1) t_{r+1}.
\end{equation}
A version of the Kruskal--Katona theorem by Lov\'asz \cite{Lov} states that given a family $\mathcal{A}$ of $k$-element sets, if $|\mathcal{A}|> \binom{x}{k}$ for some real number $x$, then the size of its shadow $\partial \mathcal{A}$ is greater than $\binom{x}{k-1}$.  
This implies that if an $r$-uniform hypergraph has at most $\binom{x}{r}$ edges, then $t_{r+1}$, the number of the $(r+1)$-cliques in the hypergraph, is at most $\binom{x}{r+1}$. Since $|E|= \rho \binom{n}{r}\le \binom{\rho^{1/r} n + r -1}{r}$, we derive that $t_{r+1}\le \binom{\rho^{1/r} n + r-1}{r+1}$. Substituting this into \eqref{eq:m}, we have that
\begin{align*}
m &\ge  \rho \binom{n}{r}( n-r ) - (r+1) \binom{\rho^{1/r} n + r-1}{r+1} \\
&=  \rho \binom{n}{r+1}( r+1 ) -  \rho^{\frac{r+1}{r}} (r+1)  \binom{n}{r+1} - O(n^r) \\
&= \rho ( 1 - \rho^{1/r} ) ( r+1) \binom{n}{r+1} -  O(n^r).
\end{align*}
Since $\rho\in [\a, 1- \a]$, we have that
\[\rho ( 1 - \rho^{1/r} )\ge \min \{ \a (1 - \a^{1/r} ), (1-\a) ( 1 - (1-\a)^{1/r} ) \}.
\]
Using the fact that $1 - \a^{1/r} \ge (1-\a)/r$, this minimum is at least $\a (1-\a)/r$. 
As $n$ is sufficiently large, this gives that
$$
m \ge \a (1-\a) \frac{r+1}{r}\binom{n}{r+1} -   O(n^r) \ge \a(1-\a)\binom{n}{r+1}.
$$
\end{proof}

\begin{prop}\label{evensum}
For $r \in \mathbb N$, $0\le c\le 1$ and $n\to \infty$, 
\begin{align*}
\sum_{0\le i\le r, \, i \, \rm{even}} \binom{cn}{r- i} \binom{ (1-c)n }{i} & = \frac{n^r}{2r!} ( 1+ (2c-1)^r ) - O(n^{r-1}), \\
\sum_{0\le i\le r, \, i \,\rm{odd}} \binom{cn}{r- i} \binom{ (1-c)n }{i} & = \frac{n^r}{2r!} ( 1- (2c-1)^r ) - O(n^{r-1}).
\end{align*}

\end{prop}
\begin{proof} Throughout the proof we assume that $0\le i\le r$. We observe that
\[
\binom{cn}{r-i} \binom{ (1-c)n }{i} = \frac{(cn)^{r-i} }{(r-i)!}\frac{(1-c)^{i} n^{i}}{ i! } - O(n^{r-1}) = \frac{(cn)^r}{r!} \binom{r}{i} \left(\frac{1-c}{c} \right)^i - O(n^{r-1}).
\]
Since 
\[
\left( \sum_{i \,\rm{even}} + \sum_{i \, \rm{odd}} \right) \binom{r}{i} \left(\frac{1-c}{c} \right)^i = \sum_i \binom{r}{i} \left(\frac{1-c}{c} \right)^i = \left(1 + \frac{1-c}{c} \right)^r = \frac{1}{c^r},  
\]
\[
\left( \sum_{i \,\rm{even}} - \sum_{i \, \rm{odd}} \right) \binom{r}{i} \left(\frac{1-c}{c} \right)^i = \sum_i (-1)^i \binom{r}{i} \left(\frac{1-c}{c} \right)^i = \left(1 - \frac{1-c}{c} \right)^r = \left(\frac{2c-1}{c}\right)^r,
\]
we have 
\begin{align*}
\sum_{i \,\rm{even}} \binom{r}{i} \left(\frac{1-c}{c} \right)^i & = \frac12 \left( \frac{1}{c^r}  + \left(\frac{2c-1}{c}\right)^r \right) \quad \text{and} \quad
\sum_{i \,\rm{odd}} \binom{r}{i} \left(\frac{1-c}{c} \right)^i & = \frac12 \left( \frac{1}{c^r}  - \left(\frac{2c-1}{c}\right)^r \right).
\end{align*}
The two desired equalities follow immediately.
\end{proof}

\section{Proof of Theorem~\ref{mainthm}}
Most of the paper is devoted to proving the following result.
\begin{thm}\label{nonexthm1}
Let $\eps >0$ and $k, \ell \in \mathbb N$ such that $k\ge 3$ and $k/2 \leq \ell \leq k-1$.
Then there exist $\a, \xi >0$ and $n_0 \in \mathbb N$ such that the following holds. 
Suppose that $H$ is a $k$-uniform hypergraph on $n \geq n_0$ vertices. If
$$\delta _{\ell} (H) \geq \left( \frac{1}{2}-\a \right) \binom{n-\ell}{k-\ell}$$
then  $H$ is $\eps$-close to $\mathcal B_{n,k}$ or $\overline{\mathcal B}_{n,k}$, or $H$ contains
a matching $M$ of size $|M| \le \xi n/k$ that absorbs any set $W\subseteq V(H) \setminus V(M)$ such that $|W| \in k\mathbb{N}$ with $|W| \le \xi^2 n$.
\end{thm}

We prove Theorem~\ref{nonexthm1} in Section~\ref{sectionnon}. Once Theorem~\ref{nonexthm1} is proven, we
can derive Theorem~\ref{mainthm} in the same way as described in~\cite{zhao}. For completeness, we include the proof
here.
\begin{thm}\cite[Theorem 4.1]{zhao}\label{extthm}
Given $1\le \ell \le k-1$, there exist $\eps > 0$ and $n_0 \in \mathbb{N}$ such that the following holds. Suppose that $H$ is a $k$-uniform hypergraph on $n\ge n_0$ vertices such that $n$ is divisible by $k$. If  $\delta_{\ell} (H) >  \d(n, k, \ell)$ and $H$ is $\eps$-close to $\mathcal B_{n,k}$ or $\overline{\mathcal B}_{n,k}$, then $H$ contains a perfect matching.
\end{thm}
Theorem~\ref{extthm} ensures a perfect matching when our hypergraph $H$ is `close' to one of the
`extremal' hypergraphs $\mathcal B_{n,k}$ and $\overline{\mathcal B}_{n,k}$. When $H$ is non-extremal
we will apply the following result of Markstr\"om and Ruci\'nski~\cite{mark} to ensure an `almost' perfect matching
in $H$.
\begin{thm}\cite[Lemma 2]{mark}\label{marklemma}
For each integer $k\ge 3$, every $1 \leq \ell \leq k-2$ and every $\gamma >0$ there exists an  $n_0 \in \mathbb N$
such that the following holds.
Suppose that $H$ is a $k$-uniform hypergraph on $n\geq n_0$ vertices such that
$$\delta _{\ell} (H) \geq \left(\frac{k-\ell}{k} - \frac{1}{k^{(k-\ell)}}+\gamma \right)\binom{n-\ell}{k-\ell}.$$
Then $H$ contains a matching covering all but at most $\sqrt{n}$ vertices.
\end{thm}
In~\cite{mark}, Markstr\"om and Ruci\'nski only stated Theorem~\ref{marklemma} for $1\leq \ell <k/2$. In fact, their proof works for all values of $\ell$ such that $1 \leq \ell \leq k-2$. In the case when $\ell = k-1$, we need a result of R\"odl, Ruci\'nski and Szemer\'edi~\cite[Fact 2.1]{rrs}: Suppose $H$ is a $k$-uniform hypergraph on $n$ vertices. If $\delta_{k-1}(H)\ge n/k$, then $H$ contains a matching covering all but at most $k^2$ vertices in $H$.

\medskip

\noindent
{\bf Proof of Theorem~\ref{mainthm}.} 
Let $\varepsilon$ be as in Theorem~\ref{extthm} and $\a, \xi$ be as in Theorem~\ref{nonexthm1}.
That is,
$$0<\a , \xi \ll \eps \ll 1/k.$$
Assume that $k/2\le \ell \le k-1$. Suppose that $n$ is sufficiently large and $k$ divides $n$. Consider any $k$-uniform hypergraph $H$ on $n$ vertices such that 
$$\delta _\ell (H) >\delta (n,k,\ell) .$$

By the definition of $\mathcal{H}_{\text{ext}} (n,k)$, there exists a $k$-uniform hypergraph $\mathcal{B} _{n,k}(A,B)\in \mathcal{H}_{\text{ext}} (n,k)$ with $|A| \in \{\lfloor n/2 \rfloor, \lfloor n/2 \rfloor +1\}$. Clearly $\delta_{k-1}(\mathcal{B}_{n, k}(A,B))\ge n/2 - k$. Thus, by Proposition~\ref{prop:deg}, $\delta_{\ell}(\mathcal{B}_{n, k}(A,B)) \geq (1/2 -\alpha)\binom{n-\ell}{k-\ell}$. Consequently $\delta _{\ell} (H) \geq (1/2- \a) \binom{n-\ell}{k-\ell}$. Theorem~\ref{nonexthm1} implies that
either $H$ is $\eps$-close to $\mathcal B_{n,k}$ or $\overline{\mathcal B}_{n,k}$ or $H$ contains
a matching $M$ of size $|M| \le \xi n/k$ that absorbs any set $W\subseteq V(H) \setminus V(M)$ such that $|W| \in k\mathbb{N}$ with $|W| \le \xi^2 n$.
In the former case Theorem~\ref{extthm} implies that $H$ contains a perfect matching. In the latter case set $H':=H[V(H)\setminus V(M)]$ and $n':=|V(H')|$.
Since $\ell\ge k/2$, $\alpha, \xi \ll 1/k$ and $n$ is sufficiently large,
$$\delta _{\ell} (H') \geq \delta _{\ell} (H) - |V(M)|\binom{n}{k-\ell-1} \geq
\left(\frac{k-\ell}{k} - \frac{1}{k^{(k-\ell)}}+\a \right)\binom{n'-\ell}{k-\ell}.$$
Therefore, if $\ell \leq k-2$, Theorem~\ref{marklemma} implies that $H'$ contains a matching $M'$ covering all but at most $\sqrt{n'}$ vertices in $H'$. If $\ell=k-1$, then since $\delta _{\ell} (H') \geq n'/k$, Fact 2.1 from~\cite{rrs} implies that $H'$  contains a matching $M'$ covering all but at most $k^2$ vertices in $H'$.
In both cases set $W:=V(H')\backslash V(M')$. Then $|W|\leq \sqrt{n'} \le \xi^2 n$. By definition of $M$, there is a matching $M''$ in $H$ which covers
$V(M)\cup W$.  Hence, $M' \cup M''$ is a perfect matching of $H$, as desired.
\endproof

\section{Outline of the proof of Theorem~\ref{nonexthm1}}\label{presec}
In~\cite{zhao} we proved Theorem~\ref{nonexthm1} in the case when $k$ is divisible by $4$. In this section
we give an overview of the proof of Theorem~\ref{nonexthm1} and explain how our method differs to that used in~\cite{zhao}.
\subsection{The method used in~\cite{zhao}}
Let $k \in \mathbb N$ be divisible by $4$. Consider a $k$-uniform hypergraph $H$ as in Theorem~\ref{nonexthm1}.
Define the graph $G'$ with vertex set $\binom{V(H)}{k/2}$ in which two vertices 
$\u{x}, \u{y} \in V(G')$ are adjacent if and only if $\u{x}\cup \u{y} \in E(H)$. 
Set $N:=|G'|$.
In~\cite{zhao} the proof splits into two main steps.

\medskip

\noindent
{\bf Step 1:} We prove that $G'$ or $\overline{G'}$ is `close' to $K_{N/2,N/2}$ or $H$ contains the matching $M$ as desired in Theorem~\ref{nonexthm1}. 

\noindent
{\bf Step 2:} If $G'$ or $\overline{G'}$ is `close' to $K_{N/2,N/2}$ then we prove that $H$ is $\eps$-close to $\mathcal B_{n,k}$ or $\overline{\mathcal B}_{n,k}$.

\medskip

\noindent
Notice that we cannot adopt quite the same approach as above to prove Theorem~\ref{nonexthm1} for \emph{all} values of $k$. Indeed,
to define $G'$ we require that $k$ is even. Moreover, the proof of Step 2 in~\cite{zhao} uses that $k$ is divisible by $4$:
Since $G'$ or $\overline{G'}$ is `close' to $K_{N/2,N/2}$, we obtain a `natural' partition $R,B$ of $V(G')=\binom{V(H)}{k/2}$
where $|R|=|B|=N/2$. Consider a complete $k/2$-uniform hypergraph $K$ whose vertex set is $V(H)$. Thus,
$E(K)=\binom{V(H)}{k/2}$. Hence, we can view the partition $R,B$ of $\binom{V(H)}{k/2}$ as a $2$-coloring of $E(K)$.
We then apply the hypergraph removal lemma (see e.g.~\cite{gowers, rodlskokan}) together with a  result of Keevash and Sudakov~\cite{keesud}
to obtain structure in $K$.
(We show that $K[R]$ or $K[B]$ is `close' to $\mathcal B_{n,k/2}$.)
 This structure in $K$ together with the fact that $G'$ or $\overline{G'}$ is `close' to $K_{N/2,N/2}$ implies that 
$H$ is $\eps$-close to $\mathcal B_{n,k}$ or $\overline{\mathcal B}_{n,k}$. Crucially, the result of Keevash and Sudakov concerns 
hypergraphs of even uniformity. Thus, we require that $K$ has even uniformity and hence, that $k$ is divisible by $4$.

\subsection{The new method}
Let $H$ be a $k$-uniform hypergraph on $n$ vertices as in Theorem~\ref{nonexthm1}. To prove Theorem~\ref{nonexthm1} we introduce a
bipartite analog of $G'$. Set $r := \lceil k/2 \rceil$, $r' := \lfloor k/2 \rfloor$, $X^r:=\binom{V(H)}{r}$ and $Y^{r'}:=\binom{V(H)}{r'}$. Further, let $N:=\binom{n}{r}$ and $N':=\binom{n}{r'}$.
Define the bipartite graph $G$ as follows:
$G$ has vertex classes $X^r$ and $Y^{r'}$. Two vertices
$x_1\dots x_r \in X^r$ and $y_1\dots y_{r'} \in Y^{r'}$ are adjacent in $G$ if and only if $x_1 \dots x_ry_1 \dots y_{r'} \in E(H)$.
The proof again splits into two main parts.
\medskip

\noindent
{\bf Step 1:} We prove that $G$ is `close' to the disjoint union of two copies of $K_{N/2, N'/2}$ or $H$ contains the matching $M$ as desired in Theorem~\ref{nonexthm1}.

\noindent
{\bf Step 2:} If $G$ is `close' to the disjoint union of two copies of $K_{N/2, N'/2}$ then we prove that $H$ is $\eps$-close to $\mathcal B_{n,k}$ or $\overline{\mathcal B}_{n,k}$.

\medskip

\noindent
Step 1 can be proved using a similar approach to the corresponding step in~\cite{zhao}.
Step 2, however, is tackled in a different way. Indeed, we do not consider an auxiliary hypergraph $K$
as in~\cite{zhao}. Instead, we obtain structure in $H$ through direct arguments on the graph $G$.

\section{Proof of Theorem~\ref{nonexthm1}}\label{sectionnon}
In this section we prove Theorem~\ref{nonexthm1}. 
Let $\alpha >0$ and $k, \ell \in \mathbb N$ such that $k\ge 3$ and $k/2 \leq \ell \leq k-1$. Given a $k$-uniform hypergraph $H$ on $n$ vertices such that $\delta _{\ell} (H) \geq \left( \frac{1}{2}-\a \right) \binom{n-\ell}{k-\ell}$, by Proposition~\ref{prop:deg}, we have $\delta _{r} (H) \geq \left( \frac{1}{2}- \a \right) \binom{n-r}{k-r}$ where $r := \lceil k/2 \rceil$. Hence, in order to prove Theorem~\ref{nonexthm1} it suffices to prove the following result.

\begin{thm}\label{thm:extabs}  
Given any $\eps >0$ and integer $k \geq 3$, there exist $\a, \xi >0$ and $n_0 \in \mathbb N$ such that the following holds. Set $r := \lceil k/2 \rceil$. Suppose that $H$ is a $k$-uniform hypergraph on $n \geq n_0$ vertices. If
$$\delta _{r} (H) \geq \left( \frac{1}{2}-\a \right) \binom{n - r}{k- r}$$
then $H$ is $\eps$-close to $\mathcal B_{n,k}$ or $\overline{\mathcal B}_{n,k}$, or $H$ contains
a matching $M$ of size $|M| \le \xi n/k$ that absorbs any set $W\subseteq V(H) \setminus V(M)$ such that $|W| \in k\mathbb{N}$ with $|W| \le \xi^2 n$.
\end{thm}
Theorem~\ref{thm:extabs} immediately follows from Lemmas~\ref{lem:abs}--\ref{lem:GH}.
The following lemma from~\cite{zhao} states that in order to find the absorbing set described in Theorem~\ref{thm:extabs}, it suffices to prove that there are at least $\xi n^{2k}$ absorbing $2k$-sets for every fixed $k$-set from $V(H)$.
\begin{lemma}\cite[Lemma 5.2]{zhao}
\label{lem:abs}
Given $0<\xi \ll 1$ and an integer $k \geq 2$,  there exists an $n_0 \in \mathbb N$ such that the following holds. Consider a $k$-uniform
hypergraph $H$ on $n \geq n_0$ vertices. Suppose that any $k$-set of vertices $Q \subseteq V(H)$ can be absorbed by at least
$\xi n^{2k}$ $2k$-sets of vertices from $V(H)$. Then $H$ contains a matching $M$ of size $|M|\leq \xi n/k$ that absorbs any set
$W \subseteq V(H) \backslash V(M)$ such that $|W|\in k \mathbb N$ and $|W|\leq \xi ^2 n$.
\end{lemma}

Throughout this section we will use the following notation.
Let $k \geq 3$ be an integer and set $r:=\lceil k/2 \rceil$ and $r':=k-r$.
Given a $k$-uniform hypergraph $H$, define $X^r:=\binom{V(H)}{r}$ and $Y^{r'}:=\binom{V(H)}{r'}$.
Set $N:=\binom{n}{r}$ and $N':=\binom{n}{r'}$.

Given a $k$-uniform hypergraph $H$, we define the bipartite graph $G(H)$ as follows:
$G(H)$ has vertex classes $X^r$ and $Y^{r'}$. Two vertices
$x_1\dots x_r \in X^r$ and $y_1\dots y_{r'} \in Y^{r'}$ are adjacent in $G(H)$ if and only if $x_1 \dots x_ry_1 \dots y_{r'} \in E(H)$.
When it is clear from the context, we will often refer to $G(H)$ as $G$.

Let $k \geq 3$ and $n$ be positive integers. 
 Denote by $B_{n,k}$ the bipartite graph with vertex classes $X$ and $Y$ of sizes $N$ and $N'$ respectively which
satisfies the following properties:
\begin{itemize}
\item $X_1,X_2$ is a partition of $X$ such that $|X_1|=\lceil N/2 \rceil$ and $|X_2|=\lfloor N/2 \rfloor$;
\item $Y_1,Y_2$ is a partition of $Y$ such that $|Y_1|=\lceil N'/2 \rceil$ and $|Y_2|=\lfloor N'/2 \rfloor$;
\item  $B_{n,k} [X_1,Y_1]$ and $B_{n,k} [X_2,Y_2]$ are complete bipartite graphs. Further, there are no other edges in $B_{n,k} $.
\end{itemize}

\begin{lemma}
\label{lem:G}
Given any $\b>0$ and an integer $k\geq 3$, there exist $\a , \xi >0$, and $n_0 \in \mathbb N$ such that the following holds.
Suppose that $H$ is a $k$-uniform hypergraph on $n \geq n_0$ vertices so that $$\d_{r}(H) \ge \left(\frac12 - \a \right)\binom{n-r}{k-r}.$$
Set $G:= G(H)$. Then at least one of the following assertions holds.
\begin{itemize}
\item $G = B_{n,k} \pm \b NN'$; in other words, $G$  becomes a copy of $B_{n,k}$ after adding or deleting at most $ \beta NN'$ edges.
\item There are at least $\xi n^{2k}$ absorbing $2k$-sets in $\binom{V(H)}{2k}$ for every $k$-subset of $V(H)$.
\end{itemize}
\end{lemma}

\begin{lemma}
\label{lem:GH}
Given any $\eps >0$ and integer $k \geq 3$, there exist $\b>0$ and $n_0 \in \mathbb N$ such that the following holds.
Suppose that $H$ is a $k$-uniform hypergraph on $n \geq n_0$ vertices. Suppose further that $G:= G(H)$ satisfies $G = B_{n,k} \pm \b NN'$. Then $H$ is $\eps$-close to $\mathcal B_{n,k}$ or $\overline{\mathcal B}_{n,k}$.
\end{lemma}
The rest of the section is devoted to the proof of Lemmas~\ref{lem:G} and \ref{lem:GH}.

\subsection{Proof of Lemma~\ref{lem:G}}
Given $\b> 0$, we choose additional constants $\g, \a, \xi$ such that
\begin{equation}
\label{eq:agxi}
0< \xi \ll \a \ll \g \ll \beta.
\end{equation}
Without loss of generality we may assume that $\b \ll 1/k$.
We also assume that $n$ is sufficiently large. 

We have that 
\begin{align}\label{minr}
\delta _{r} (H) \geq \left( 1/2 -\alpha \right )\binom{n-r}{k-r} \geq \left( 1/2 -2 \alpha \right )\binom{n}{r'}
\end{align}
and so by Proposition~\ref{prop:deg},
\begin{align}\label{minr'}
\delta _{r'} (H) \geq \left( 1/2 -\alpha \right )\binom{n-r'}{k-r'} \geq \left( 1/2 -2 \alpha \right )\binom{n}{r}.
\end{align}

Let $Q\subseteq V(H)$ be a $k$-set.
It is easy to see that if $Q$ has at least $\g^3 n^{k}$ absorbing $k$-sets then $Q$ has at least $\xi n^{2k}$ absorbing $2k$-sets. Indeed, let $P$ be an absorbing $k$-set for $Q$. Then $P\cup e$ is an absorbing $2k$-set for $Q$ for any edge $e\in E(H - (P\cup Q))$. Note that
$$|E(H)| \ge \left(\frac12 - \a \right)\binom{n-r}{k-r}\times \frac{\binom{n}{r}}{\binom{k}{r}} = \left( \frac{1}{2}- \a \right) \binom{n}{k}.$$
Thus, as $n$ is sufficiently large, there are at least
\[
\left(\frac12 - \a \right)\binom{n}{k} - 2k\binom{n}{k-1} \ge \frac{n^{k}}{4( k!)}
\]
edges in $H- (P\cup Q)$. Since an absorbing $2k$-set may be counted at most $\binom{2k}{k}$ times when counting the number of $P, e$, there are at least
\[
\g^3 n^{k} \times  \frac{n^{k}}{4(k!)} \times\frac{1}{\binom{2k}{k}} \stackrel{(\ref{eq:agxi})}{\ge}  \xi n^{2k}
\]
absorbing $2k$-sets for $Q$.

Therefore, in order to prove Lemma~\ref{lem:G}, it suffices to prove the following two claims.

\begin{claim}\label{clm:ab}
If either of the following cases holds, then we can find $\g^3 n^{k}$ absorbing $k$-sets or $\g^3 n^{2k}$ absorbing $2k$-sets for every
$k$-set $Q\in \binom{V(H)}{k}$.
\begin{description}
\item[Case (a) ] For any $r$-tuple $\underline{a} \in \binom{V(H)}{r}$, there are at least $(\frac{1}{2} + \g) \binom{n}{r}$
$r$-tuples $\underline{b} \in \binom{V(H)}{r}$ such that $| N_H( \underline{a} ) \cap N _H (\underline{b}) | \ge \g \binom{n}{r'}$.
\item[Case (b)] $| \{ \underline{a} \in \binom{V(H)}{r'}: d_H(\underline{a})\ge (\frac12 + \g)\binom{n}{r} \} |\ge 2\g \binom{n}{r'}$.
\end{description}
\end{claim}

\begin{claim}\label{clm:abG}
If neither Case (a) or Case (b) holds, then $G = B_{n,k} \pm \b NN'$.
\end{claim}

\noindent
{\bf Proof of Claim~\ref{clm:ab}.} We argue in a similar way to the proof of Claim 5.5 in~\cite{zhao}.
Given a $k$-set $Q= \{x_1, \dots , x_{r'} , y_1 , \dots , y_r \} \subseteq V(H)$, we will consider two types of absorbing sets for $Q$:
\begin{description}
\item[Absorbing $k$-sets] These consist of a single edge $x'_1\dots x'_r y'_1 \dots y'_{r'} \in E(H)$ with the property that
both $x_1 \dots x_{r'} x'_1 \dots x'_r$ and $y_1 \dots y_r y'_1 \dots y'_{r'}$ are edges of $H$.

\item[Absorbing $2k$-sets] These consist of distinct vertices $x'_1, \dots, x'_{r}$, $y'_1 , \dots , y'_{r'}$, $w'_1, \dots , w'_{r'} $, \newline $z'_1 , \dots , z'_r \in V(H)$ such that $x'_1 \dots x'_r w'_1 \dots w'_{r'}\,$, $y'_1 \dots y'_{r'} z'_1 \dots  z'_r$ and $ w'_1 \dots w'_{r'} z'_1 \dots  z'_r$ are edges in $H$. Furthermore, $x_1 \dots x_{r'} x'_1 \dots x'_r$ and $y_1 \dots y_r y'_1 \dots y'_{r'}$ are also edges of $H$ (see Figure~1).
\end{description}
\begin{figure}\label{picture}
\begin{center}\footnotesize
\includegraphics[width=1\columnwidth]{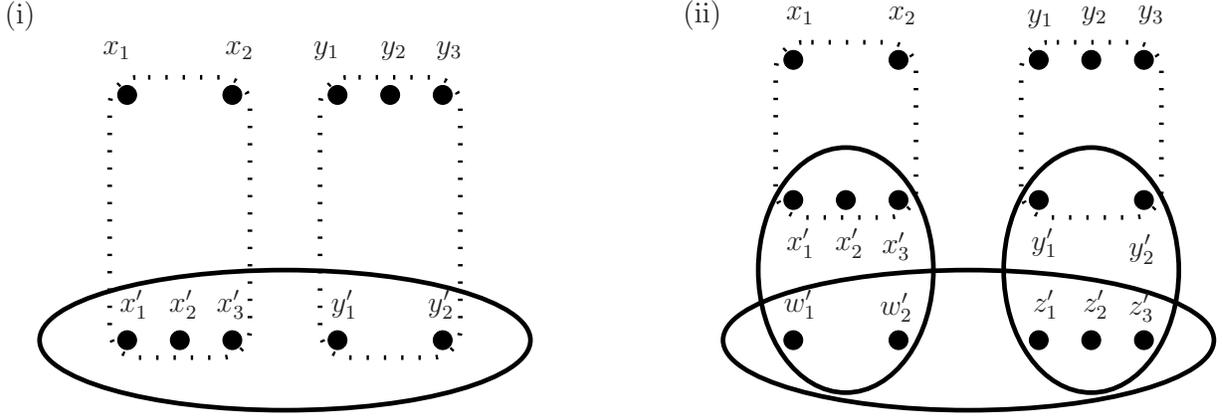}
\caption{The (i) absorbing $k$-set and (ii) absorbing $2k$-set in the case when $k=5$.}
\end{center}
\end{figure}
Write $\underline{x}:=x_1 \dots x_{r'}$ and $\underline{y}:=y_1 \dots y_r$. For any two (not necessarily disjoint) $r$-tuples $\underline{a}$, $\underline{b} \in \binom{V(H)}{r}$ we call $\underline{a}$
a \emph{good $r$-tuple for $\underline{b}$} if $|N_H (\underline{a}) \cap N_H (\underline{b})|\geq \g \binom{n}{r'}/2$.
 We first observe that $Q$ has at least $ \g^3 n^{k}$ absorbing $k$-sets if there are
\begin{align}\label{eq:good}
  \text{at least $\frac{\g}{2} \binom{n}{r}$ good $r$-tuples in $N_H(\underline{x})$ for $\underline{y}$}.
\end{align}
Indeed, assume that (\ref{eq:good}) holds. 
There are at most $r\binom{n}{r-1}$ $r$-tuples in $\binom{V(H)}{r}$ that contain at least one  element from $\{y_1, \dots , y_r \}$.
Therefore, there are at least $\g \binom{n}{r}/2 - r\binom{n}{r-1}$ $r$-tuples in $N_H(\underline{x})$ that are good for $\underline{y}$ and disjoint from $\underline{y}$. Let us pick such an $r$-tuple $\underline{x}'=x'_1 \dots x'_r$. Thus,
$|N_H(\underline{x}')\cap N_H(\underline{y})|\ge \g \binom{n}{r'}/2$. We pick $\underline{y}'=y'_1\dots y'_{r'} \in
N_H(\underline{x}')\cap N_H(\underline{y})$ such that $\underline{y}'$ is disjoint from $\underline{x}$.
Note that there are at least $\g \binom{n}{r'}/2 - r'\binom{n}{r'-1}$ choices for $\underline{y}'$.
Notice that the $k$-set $\{x'_1, \dots , x'_r, y'_1 , \dots , y'_{r'} \}$ is an absorbing set for $Q$ since
$x'_1\dots x'_r y'_1 \dots y'_{r'} \,$,  $x_1 \dots x_{r'} x'_1 \dots x'_r$ and $y_1 \dots y_r y'_1 \dots y'_{r'}$ are edges in $H$.
Since an absorbing $k$-set may be counted $\binom{k}{r}$ times, this argument implies that there are at least
\[
\left( \frac{\g}{2} \binom{n}{r} - r\binom{n}{r-1} \right) \left( \frac{\g}{2} \binom{n}{r'} - r'\binom{n}{r'-1} \right) \frac{1}{\binom{k}{r}} \ge \g^3  n^{k}
\]
absorbing $k$-sets for $Q$. 

\medskip
Now assume that Case (a) holds. This implies that there are at least $(\frac{1}{2} + \g) \binom{n}{r}$  good $r$-tuples for
$\underline{y}$.
 By (\ref{minr'}), $d_H (\underline{x})\ge  (1/2 -2 \alpha )\binom{n}{r}$. So there are at least
 $(\g - 2\a)\binom{n}{r}\ge \g \binom{n}{r}/2$ $r$-tuples in $N_H(\underline{x})$ that are good for $\underline{y}$.  Thus, \eqref{eq:good} holds and consequently $Q$ has at least $\g^3 n^{k}$ absorbing $k$-sets.

\medskip
Next assume Case (b) holds. For any two (not necessarily disjoint) $r'$-tuples $\underline{a}$, $\underline{b} \in \binom{V(H)}{r'}$ we call $\underline{a}$
a \emph{good $r'$-tuple for $\underline{b}$} if $|N_H (\underline{a}) \cap N_H (\underline{b})|\geq \g \binom{n}{r}/2$.
By arguing in an identical fashion as before, note that
 $Q$ has at least $ \g^3 n^{k}$ absorbing $k$-sets if there are
\begin{align}\label{eq:good2}
  \text{at least $\frac{\g}{2} \binom{n}{r'}$ good $r'$-tuples in $N_H(\underline{y})$ for $\underline{x}$}.
\end{align}

Let $\Lambda:=\{ \underline{a} \in \binom{V(H)}{r'}: d_H(\underline{a})\ge (\frac12 + \g)\binom{n}{r} \}$.
So by assumption, $|\Lambda |\ge 2\g \binom{n}{r'}$. 
Note that every $\u{a} \in \Lambda $ is good for arbitrary $\u{b} \in \binom{V(H)}{r'}$ since
$|N_H (\u{a})|\geq (1/2+\gamma )\binom{n}{r}$, $|N_H (\u{b})|\geq (1/2- 2 \alpha )\binom{n}{r}$ and therefore
$|N_H (\u{a}) \cap N_H (\u{b})| \geq (\gamma -2\alpha)\binom{n}{r} \geq \gamma \binom{n}{r}/2$.
Thus, if $|\Lambda \cap N_H (\u{y})|\geq \gamma \binom{n}{r'}/2$ then (\ref{eq:good2}) is satisfied.
Therefore, we may assume that $|\Lambda \cap N_H (\u{y})|< \gamma \binom{n}{r'}/2$.

We also assume that \eqref{eq:good} fails (otherwise we are done). 
Thus, less than $\gamma \binom{n}{r}/2$ $r$-tuples in $N_H (\u{x})$ are good for $\u{y}$ and consequently,
at least $(\frac{1}{2}-2\alpha)\binom{n}{r}-\frac{\g}{2}\binom{n}{r}$ $r$-tuples $\u{x'} \in N_H (\u{x})$
satisfy $|N_H(\underline{x}')\cap N_H(\underline{y})|< \g \binom{n}{r'}/2$.
We pick such an $r$-tuple $\underline{x}'$ that is disjoint from $\underline{y}$; there are at least $(\frac12 - 2\a)\binom{n}{r} - \frac{\g}{2} \binom{n}{r} - r\binom{n}{r-1}\ge (\frac12 - \g)\binom{n}r$ $r$-tuples with this property. Since
\[
| N_H(\underline{x}') \cup N_H(\underline{y}) | \ge 2 \left(\frac12 - 2\a \right)\binom{n}{r'} - \frac{\g}{2} \binom{n}{r'}
 \ge \binom{n}{r'} - {\g} \binom{n}{r'},
\]
it follows that
\begin{align}
| \Lambda \cap N_H(\underline{x}')| & \ge |\Lambda| - | \Lambda \cap N_H(\underline{y})| - | \overline{(N_H(\underline{x}')
\cup N_H(\underline{y}))} | \nonumber \\
& \ge  2\g \binom{n}{r'} - \frac{\g}{2} \binom{n}{r'} - {\g} \binom{n}{r'} = \frac{\g}2 \binom{n}{r'}. \label{eq;x'y'}
\end{align}
Now pick any $\underline{w}' \in \Lambda \cap N_H(\underline{x}')$ that is disjoint from $Q$.
(Note there are at least $\frac{\g}2 \binom{n}{r'} - k\binom{n}{r'-1} \ge \frac{\g}{3} \binom{n}{r'}$ choices for $\underline{w}'$.)
Next pick an $r'$-tuple $\underline{y}'\in N_H(\underline{y})$ such that $\underline{y}'$ is disjoint from
$\underline{x}$, $\underline{x}'$ and $\underline{w}'$. (There are at least $(\frac12 - 2\a) \binom{n}{r'} - 2k\binom{n}{r'-1}
\ge (\frac12 - \g) \binom{n}{r'}$ choices for $\underline{y}'$ here.)
By the definition of $\Lambda$,  there are at least $(\g - 2\a) \binom{n}{r}$ $r$-sets in $N_H(\underline{w}'
)\cap N_H(\underline{y}')$.
We pick $\underline{z}'\in N_H(\underline{w}')\cap N_H(\underline{y}')$  such that $\underline{z}'$ is disjoint from
$\underline{x}$, $\underline{y}$ and $\underline{x}'$. (There are at least $(\g - 2\a) \binom{n}{r} - 2k\binom{n}{r-1}
\ge \g  \binom{n}{r}/2$ choices for $\underline{z}'$ here.)

Let $S$ denote the $2k$-set consisting of the vertices contained in $\underline{x}'$, $\underline{y}'$, $\underline{w}'$
 and $\underline{z}'$. By the choice of $\underline{x}'$, $\underline{y}'$, $\underline{w}'$
 and $\underline{z}'$, $S$ is an absorbing $2k$-set for $Q$.

In summary, there are at least $(\frac12 - \g) \binom{n}{r} $ choices for $\underline{x}'$, at least
$\frac{\g}{3} \binom{n}{r'} $ choices for $\underline{w}'$, at least $(\frac12 - \g) \binom{n}{r'}$ choices for
$\underline{y}'$  and at least $\frac{\g}{2} \binom{n}{r}$ choices for $\underline{z}'$. Since each absorbing $2k$-set may be counted $\binom{2k}{r} \binom{2k-r}{r'} \binom{k}{r}$ times, there are at least
\[
 \left(\frac12 - \g \right)^2 \binom{n}{r} \binom{n}{r'} \frac{\g}3 \binom{n}{r'} \frac{\g}{2} \binom{n}{r}
\times \frac{1}{\binom{2k}{r} \binom{2k-r}{r'} \binom{k}{r}}
 \stackrel{(\ref{eq:agxi})}{\geq } \g ^3 n^{2k}
\]
absorbing $2k$-sets for $Q$, as desired.
\endproof

\noindent
{\bf Proof of Claim~\ref{clm:abG}.}  Note that by (\ref{minr}) and (\ref{minr'}),
\begin{align}\label{Gmin}
d_G (\u{x}) \geq (1/2-\gamma )N' \text{ for all } \u{x} \in X^r \ \text{ and } \ d_G (\u{y}) \geq (1/2-\gamma )N \text{ for all } \u{y} \in Y^{r'}.
\end{align}
Further by assumption, the following conditions hold.
\begin{description}
\item[(i)] There exists a vertex $ \u{a} \in X^r$ such that at most $(\frac{1}{2} + \gamma) N$
vertices $\u{b} \in X^r$ satisfy $| N_G(\u{a}) \cap N_G(\u{b}) | \ge \gamma N'$.

\item[(ii)] $| \{ \u{v} \in Y^{r'} : d_G(\u{v})\ge (\frac12 + \gamma)N \} |< 2\gamma N'$.
\end{description}

Let $B':= N_G(\u{a})\subseteq Y^{r'}$ and 
$A'':=\{ \u{x} \in X^r: | B' \cap N_G (\u{x}) | < \g N' \}$.  Then $|B'| \ge (\frac12 - \gamma)N'$ and $|A''|\ge (\frac12 - \g) N$.

We also need an upper bound on $|B'|$. Fix $\u{x} \in A''$. Since $|N_G(\u{x})| \ge (\frac12 - \gamma) N'$,  we have
\[
|B'| + \left(\tfrac12 - \gamma \right) N' \le |B'| + |N_G(\u{x})| = |B'\cup N_G(\u{x})| + |B' \cap N_G (\u{x})| \le N' + \g N',
\]
which gives $|B'|\le (\tfrac12 + 2 \g) N'$.

By the definition of $A''$, we have $e(A'', B')\le \g N' |A''| $. Thus, (\ref{Gmin}) implies  that
\begin{equation}\label{eq:ABbar}
e (A'', B'')\ge (1/2 - 2\g) N'  |A''|,
\end{equation}
where $B'':= Y^{r'}\setminus B'$. Next we show that $e(A', B'')$ is very small, where $A':=X^r \setminus A''$.

\begin{claim}\label{clm:AcB}
$e(A', B'')\le 8 \sqrt{\g} |A'| |B''|$.
\end{claim}
\begin{proof}
Assume for a contradiction that the claim is false.
Set $B_1 := \{\u{x} \in B'':  d_G(\u{x}, A') \ge 4\sqrt{\g} |A'| \}$. By assumption
\[
8\sqrt{\g} |A'| |B''| \le e(A' , B'' ) \le |B_1| |A'| +4\sqrt{\g}|A'|  |B''|,
\]
which gives that $|B_1|\ge 4\sqrt{\g} |B''|$. By \eqref{eq:ABbar}, as $|B''|\le (\frac12 + \g) N' $, we derive that
\begin{equation}\label{eq:AcB}
e(A'', B'')\ge (\tfrac12 - 2 \g) N' |A''| \ge (1-6\g)(\tfrac12+ \g) N' |A''| \ge (1 - 6\g) |A''| |B''|.
\end{equation}
Let $B_2 := \{ \u{x} \in B'' :  d_G(\u{x}, A'') \ge (1- 3\sqrt{\g}) |A''| \}$. We claim that $|B_2|\ge (1 - 3\sqrt{\g}) |B''| $. 
Indeed,  consider $\bar{e}(A'', B''):=|A''||B''|-{e}(A'', B'')$.
If $|B_2|< (1 - 3\sqrt{\g}) |B''| $, 
then $\bar{e}(A'', B'')\ge 3\sqrt{\g} |B''| 3\sqrt{\g} |A''| = 9\g |A''| |B''|$, contradicting \eqref{eq:AcB}.

Let $B_0 := B_1 \cap B_2$. We have that $|B_0|\ge (4\sqrt{\g} - 3\sqrt{\g}) |B''|$. Since $|B''|\ge N' /3$ and $\g \le 1/36$,  we derive that $|B_0| \ge \sqrt{\g}N'/3\ge 2\g N'$. For every $\u{x}\in B_0$,
we have
\begin{align*}
d_G(\u{x}) & =  d_G(\u{x}, A'') + d_G(\u{x}, A') \\
& \ge (1- 3\sqrt{\g}) |A''| + 4\sqrt{\g} |A'| = (1- 7\sqrt{\g}) |A''| + 4\sqrt{\g} N \\
& \ge \left(\frac12 - \frac{7}{2}\sqrt{\g} + 4\sqrt{\g} - \g \right)N
\ge  \left(\frac12 + \g \right)N.
\end{align*}
(The penultimate inequality follows since $ |A''|\ge  (\tfrac12 - \g) N$.)
This is a contradiction to (ii).
\end{proof}
Recall that $e(A'', B') \leq \gamma N' |A''|\leq \gamma NN'$. Thus, by (\ref{Gmin})
\begin{align}\label{eqnew}
 e(A',B')=e(X^r,B') - e(A'',B') \geq
  (1/2-\gamma )N|B'| -\g NN' \geq (1/2-4 \gamma )N|B'|.
\end{align}
 (The last inequality follows since $|B'|\geq (1/2-\g)N'$.)
 Therefore, as $e(A',B') \leq |A'||B'|$, we have  $|A'|\geq (1/2-4\g)N$.

Pick a set $X' \subseteq X^r$ of size $\lceil N/2 \rceil$ such that $|X' \cap A'|$ is maximized.
Similarly, pick a set $Y' \subseteq Y^{r'}$ of size $\lceil N'/2 \rceil$ such that $|Y' \cap B'|$ is maximized.
Set $X'':= X^r \setminus X'$ and $Y'':=Y^{r'} \setminus Y'$. 
\begin{claim}\label{targets}
The following conditions hold:
\begin{itemize}
\item $e(X',Y')\geq |X'||Y'|-\beta NN'/4$;
\item $e(X'',Y'')\geq |X''||Y''|-\beta NN'/4$;
\item $e(X',Y''),e(X'',Y') \leq \beta NN'/4$.
\end{itemize}
\end{claim}
\proof
Since $|A' \cap X'| \ge (1/2 - 4\gamma) N$, we have $|A'' \cap X'| \le \lceil N/2 \rceil
- (1/2 - 4\gamma) N < 5\gamma N$.
 Further, $|B' \cap Y'|\geq (1/2-\g )N'$ and so $|B'' \cap Y'| \le \lceil N'/2 \rceil
- (1/2 - \gamma) N' < 2\gamma N'$.
 Hence, 
 \begin{align*}
 e(X'',Y'') & \geq e(A'',B'')-|A'' \setminus X''|N'-|B'' \setminus Y''|N \stackrel{(\ref{eq:ABbar})}{\geq} (1/2 -2 \g )N'|A''| -7 \g NN'
 \\ & \geq (1/4 -9 {\g})NN' \geq |X''||Y''|- \b NN' /4.
 \end{align*}
 (The penultimate inequality follows as $|A''|\geq (1/2-\g )N$.)
 Note that $|A'\setminus X' |\leq \g N$ as $|A'|\leq (1/2+\g )N$ and
 $|B' \setminus Y'|\leq 2 \g N'$ as $|B'|\leq (1/2+2\g)N'$. Thus, 
  \begin{align*}
 e(X',Y') & \geq e(A',B')-|A' \setminus X'|N'-|B' \setminus Y'|N \stackrel{(\ref{eqnew})}{\geq} (1/2 - 4\g )N|B'| 
 -3 \g NN'
 \\ & \geq (1/4 -6{\g})NN' \geq |X'||Y'|- \b NN' /4.
 \end{align*}
  (The penultimate inequality follows as $|B'|\geq (1/2-\g )N'$.)
 Since  $|A'|\geq (1/2-4\g)N$, $|X'\setminus A'|\leq 5 \g N$. Further, $|B''|\geq (1/2-2\g)N'$ and so
 $|Y''\setminus B''|\leq 2 \g N'$. Thus, by Claim~\ref{clm:AcB},
 $$e(X',Y'') \leq e(A',B'') +|X' \setminus A'|N'+ |Y''\setminus B''|N \leq 8\sqrt{\g} |A'||B''|+7\gamma NN' \leq \b NN'/4.$$
 Since $|A''|\geq (1/2-\g)N$, $|X''\setminus A''|\leq \g N$. Further, $|B'|\geq (1/2-\g )N'$ and so 
 $|Y' \setminus B'|\geq 2\gamma N'$.
 Hence,
 $$e(X'',Y') \leq e(A'',B')+|X''\setminus A''|N'+|Y' \setminus B'|N \leq \g N'|A''|+3 \g NN' \leq \b NN'/4.$$
 \endproof
Claim~\ref{targets} immediately implies that $G=B_{n,k} \pm \b NN'$, as desired.
\endproof

\subsection{Proof of Lemma~\ref{lem:GH}}\label{lastlemma}
Define constants $\beta, \beta _1, \eta$ and $n_0 \in \mathbb N$ so that
\begin{align}\label{hier2}
0< 1/n_0 \ll \beta \ll \beta _1 \ll \eta \ll \eps, 1/k.
\end{align}
Let $H$ and $G$ be as in the statement of the lemma. 
Throughout this section, when it is clear from the context, we will write $V$ for the vertex set $V(H)$ and $E$ for the edge set $E(H)$. 
Note that $r\geq 2 $ and $r'\ge 1$ as $k \geq 3$.
Since $G$ is a bipartite graph with vertex classes $X^r$ and $Y^{r'}$, and $G=B_{n,k} \pm \beta NN'$, there exists a partition
 $ X_1 , X_2$ of $X^r$ and a partition $ Y_1 , Y_2$ of $Y^{r'}$ so that  
\begin{itemize}
\item $|X_1| = \lceil N/2 \rceil $, $|X_2| = \lfloor N/2 \rfloor$, $ |Y_1| = \lceil N'/2 \rceil $, and $ |Y_2| = \lfloor N'/2 \rfloor $;
\item $ | E(G) \triangle E(K_{X_1, Y_1} \cup K_{X_2, Y_2}) | \le \beta N N'$, in other words, $G$ becomes the disjoint union of two complete bipartite graphs $K_{X_1, Y_1}$ and  $K_{X_2, Y_2}$ after adding or removing at most $\beta N N'$ edges.  
\end{itemize}
Throughout the proof, we assume that $B_{n, k}= K_{X_1, Y_1} \cup K_{X_2, Y_2}$.

We call a $k$-subset $S$ of $V$ \emph{bad} if there are two partitions $ P, P'$ and  $Q , Q'$ of $S$ such that both of the following
conditions hold:
\begin{enumerate}
\item $PP'\in E(B_{n,k})$, namely, either $P\in X_1$ and $P'\in Y_1$ or $P\in X_2$ and $P'\in Y_2$;
\item $QQ'\not\in E(B_{n,k})$, namely, either $Q\in X_1$ and $Q'\in Y_2$ or $Q\in X_2$ and $Q'\in Y_1$.
\end{enumerate}

\begin{claim}\label{badclaim}
At most $\beta  N N'$ $k$-subsets of $V$ are bad.
\end{claim}

\begin{proof}
Let $S\subseteq V$ be a bad $k$-set and let $P,P'$ and $Q,Q'$ be partitions of $S$ as in the definition of a bad $k$-set. 
If $S\in E(H)$, then $QQ' \in E(G)\setminus E(B_{n,k})$; otherwise $PP'\in E(B_{n,k})\setminus E(G)$. In either case $S$ gives rise to an edge from $E(G) \triangle E(B_{n,k})$. Furthermore, two different bad $k$-sets give two different edges of $E(G) \triangle E(B_{n,k})$. Since $|E(G)\triangle E(B_{n,k})| \le \beta N N'$, the number of bad $k$-sets is at most $\beta N N'$.
\end{proof}

Viewing $X_1, X_2$ as the colors of $r$-subsets of $V$, we define the color function $\phi: X^{r} \to \{X_1, X_2\}$ with $\phi(P)= X_i$ if $P\in X_i$. Similarly we define $\psi: Y^{r'} \to \{Y_1, Y_2\}$ such that $\psi(Q)= Y_i$ if $Q\in Y_i$. Given two distinct vertices $u, v\in V$, we define two symmetric functions\footnote{It looks simpler to define $C, D, C', D'$ functions as the families of $(r-1)$-sets or $(r'-1)$-sets instead, e.g., $C'(u,v)= \{ T\in \binom{V\setminus \{u,v\} }{r' -1}: \psi(T+u) = \psi(T+v) \}$. However, when $k=3$ (thus $r'=1$), this new definition defines $C'(u,v)= \emptyset$ for all $u, v\in V$, while our present definition of $C'(u,v)$ returns $\{uv\}$ when $\psi(u)= \psi(v)$ and $\emptyset$ otherwise.}
\begin{align*}
C(u, v) & := \left\{ S\in \binom{V}{r+1}: u, v\in S, \phi(S- u) = \phi(S-v) \right\}, \\
D(u, v) &:= \left\{ S\in \binom{V}{r+1}: u, v\in S, \phi(S- u) \ne \phi(S-v) \right\}. 
\end{align*}
Similarly, we define 
\begin{align*}
C'(u, v) & := \left\{ S\in \binom{V}{r'+1}: u, v\in S, \psi(S- u) = \psi(S-v) \right\}, \\
D'(u, v) &:= \left\{ S\in \binom{V}{r'+1}: u, v\in S, \psi(S- u) \ne \psi(S-v) \right\}. 
\end{align*}
Note that the definition of $D(u,v)$ is \emph{different} to the definition of $D_F (u,v)$ stated in Section~\ref{usesec}. Thus, when referring to the latter parameter we never omit the subscript.
Note that
\begin{equation}
\label{eq:CD}
|C(u, v)| + | D(u, v)| = \binom{n-2}{r-1} \quad \text{and} \quad |C'(u, v)| + | D'(u, v)| = \binom{n-2}{r'-1}. 
\end{equation}

\begin{claim}\label{clm:X=Y}
All but at most $\beta_1  n^{2}$ pairs of vertices $u, v\in V$ satisfy the following two properties.
\begin{description}
\item[(i)] $ |C(u, v)| \le \beta_1 n^{r-1}$ if and only if $ |C'(u, v) | \le \beta_1 n^{r'-1}$.
\item[(ii)] $ |D(u, v)| \le \beta_1 n^{r-1}$ if and only if $|D'(u, v) |\le \beta_1 n^{r'-1}$.
\end{description}
\end{claim}

\begin{proof}
Suppose for a contradiction that the claim is false.
Consider two vertices $u,v \in V$ such that (i) fails. Suppose that $|C(u, v)| \le \beta_1 n^{r-1} $ but  $|C'(u, v)| > \beta_1 n^{r'-1}$ (the other case can be proven analogously). By \eqref{eq:CD}, we have $|D(u,v)| \ge \binom{n-2}{r-1} - \beta_1 n^{r-1}$. 
We claim that $P\cup Q$ is a bad $k$-set for all $P\in D(u, v)$ and $Q\in C'(u,v)$ such that $P\cap Q = \{u, v\}$. Indeed, $P\in D(u, v)$ implies that one of $P-u$ and $P-v$ is in $X_1$ and the other is in $X_2$, and $Q\in C'(u, v)$ implies that $Q-u$ and $Q-v$ are both in $Y_i$ for some $i\in \{1, 2\}$.
If $P\cap Q = \{u, v\}$, then precisely one of the two pairs $\{P-u, Q-v\}$ and $\{P-v, Q-u\}$ is in $E(B_{n,k})$. 
Thus, by definition, the $k$-set $P\cup Q$ is bad.

Next consider a pair of vertices $u, v \in V$ that fails (ii). Suppose that $| D(u, v)| \le \beta_1 n^{r-1} $ but  $D'(u,v) >  \beta_1 n^{r'-1}$ (again, the other case can be proven analogously). By \eqref{eq:CD}, $|C(u,v)| \ge \binom{n-2}{r-1} - \beta_1 n^{r-1}$. A similar argument to before yields that $P'\cup Q'$ is a bad $k$-set for all $P'\in C(u, v)$ and $Q'\in D'(u, v)$ such that $P'\cap Q' = \{u, v\}$.

Note that given any $(r'-1)$-set, at most $(r'-1)\binom{n}{r-2}$ $(r-1)$-subsets of $V$ are \emph{not} disjoint from this set. Thus, given any $(r'+1)$-set $Q$ that contains $u$ and $v$, at most  $(r'-1)\binom{n}{r-2}$ $(r+1)$-sets $P$ that contains $u$ and $v$ satisfy $P\cap Q \ne \{u, v\}$. Further, note that $\binom{n-2}{r-1} - \beta_1 n^{r-1} - (r'-1)\binom{n}{r-2} \geq  \beta_1 n^{r-1}$.
Therefore,  
when considering all possible pairs of vertices $u, v \in V$ that fail (i) or (ii), we obtain at least 
\[
\frac{\beta_1 n^2 \times   \beta_1 n^{r-1} \times \beta_1 n^{r'-1}} {\binom{k}{2}\binom{k-2}{r-1}} \stackrel{(\ref{hier2})}{>} \beta \binom{n}{r} \binom{n}{r'} 
\]
bad $k$-sets as  a $k$-set may be counted at most $\binom{k}{2} \binom{k-2}{r-1}$ times. This contradicts Claim~\ref{badclaim}. 
\end{proof}

We call an (unordered) pair of vertices $u, v \in V$ \emph{consistent} if $u,v$ satisfy both (i) and (ii) from Claim~\ref{clm:X=Y}. Thus, all but at most $\beta_1  n^{2}$ pairs of vertices from $V$  are consistent.
We call two vertices $u, v \in V$ \emph{similar} if  $|C(u, v)| \le \beta _1 n^{r-1}$ or $| D(u, v) | \le \beta _1 n^{r-1}$.
\begin{claim}\label{clm:sim}
Less than $\beta_1 n^2$ pairs of vertices $u, v \in V$ are consistent but not similar.
\end{claim}
\begin{proof} 
Let $u,v \in V$ be consistent but not similar. Thus, $|C(u, v)| > \beta _1 n^{r-1}$ and $| D(u, v) | > \beta _1 n^{r-1}$. Since $u, v$ are consistent, $ |C(u, v)| > \beta_1 n^{r-1}$ implies that $ |C'(u, v) | > \beta_1 n^{r'-1}$.  As seen in the proof of Claim~\ref{clm:X=Y},  $P\cup Q$ is a bad $k$-set for all $P\in D(u, v)$ and $Q\in C'(u,v)$ such that $P\cap Q = \{u, v\}$.  Thus, if there are at least $\beta_1 n^2$ pairs of vertices $u, v \in V$ that are not similar but are consistent, then the number of bad $k$-subsets of $V$ is at least
\[
\frac{ \beta_1 n^2 \left( \beta_1 n^{r-1} - (r'-1)\binom{n}{r-2} \right) \beta_1 n^{r'-1} } {\binom{k}{2} \binom{k-2}{r-1}} 
\stackrel{(\ref{hier2})}{>} \beta  \binom{n}{r} \binom{n}{r'},
\]
contradicting Claim~\ref{badclaim}. 
\end{proof}

Let $v_0\in V$ be a vertex such that at least $(1- 4\beta_1)n$ vertices of $V$ are both consistent and similar to $v_0$: such 
a vertex $v_0$ exists because otherwise at least $4\beta_1 n^2/ 2 = 2\beta_1 n^2$ pairs of vertices are not consistent or
are not similar, contradicting Claim~\ref{clm:X=Y} or Claim~\ref{clm:sim}.  Let $V_0$ be the set of vertices in $V$ that are not consistent or not  similar to $v_0$. (Note that $v_0 \in V_0$.) The choice of $v_0$ implies that $|V_0|\le 4\beta_1 n$.  Define
\[
V_1 := \{ v\in V\setminus V_0:  | D(v, v_0)| \le \beta_1 n^{r-1} \}  \text{  and  } 
V_2 := \{ v\in V \setminus V_0:  | C(v, v_0)| \le \beta_1 n^{r-1} \}.
\]
Note that $V_1\cap V_2 = \emptyset$ otherwise by \eqref{eq:CD}, it implies that $\binom{n-2}{r-1} \le 2\beta_1 n^{r-1}$, a contradiction.

\begin{claim}\label{ab} The following properties hold.
\begin{itemize}
\item[(a)] $| D (v, v') | \le 3\b_1 n^{r-1}$ and $| D' (v, v') | \le 3\b_1 n^{r'-1}$  for all pairs of distinct
vertices $v,v' \in V$ such that $v, v'\in V_1$ or $v, v'\in V_2$;  
\item[(b)] $| C (v, v') | \le 3\b_1 n^{r-1}$ and $| C' (v, v') | \le 3\b_1 n^{r'-1}$  for all $v\in V_1$ and $v'\in V_2$.
\end{itemize}
\end{claim}
\proof
Let $\mathcal A$ denote the set of $(r-1)$-subsets of $V$ that contain $v_0$ and let 
$\mathcal A'$ denote the set of $(r'-1)$-subsets of $V$ that contain $v_0$. 
Consider distinct $v,v' \in V-v_0$ and suppose $T\cup \{ v,v' \} \in D(v, v')$ for some $(r-1)$-subset $T$ of $V- \{v,v' \}$. Then either $T\in \mathcal{A}$ or $T\subseteq V- \{v_0, v,v' \}$.
In the latter case, as $T+v$ and $T+v'$ have different colors, one of them has the same color as $T+v_0$ and the other has a different color to $T+v_0$. Thus, either (i) $T\cup \{v,v_0\} \subseteq D(v,v_0)$ 
 and $T \cup \{v', v_0\} \subseteq C (v',v_0)$ or (ii) $T\cup \{v,v_0\} \subseteq C(v,v_0)$ 
 and $T \cup \{v', v_0\} \subseteq D (v',v_0)$.
 This implies that
\begin{equation}
\label{eq:DDD}
|D(v, v')|\leq | D(v, v_0)| +| D(v', v_0)| + |\mathcal A| \quad \text{and} \quad |D(v, v')|\leq |C(v, v_0)| + |C(v', v_0) |+|\mathcal A|
\end{equation}
for all distinct $v,v' \in V$.
An analogous argument implies that
\begin{equation}
\label{eq:DDD2}
|D'(v, v')|\leq | D'(v, v_0)| +| D'(v', v_0)| + |\mathcal A'| \quad \text{and} \quad |D'(v, v')|\leq |C'(v, v_0)| + |C'(v', v_0) |+|\mathcal A'|
\end{equation}
for all distinct $v,v' \in V$.

Consider any distinct $v, v'\in V_1$. By the definition of $V_1$ and \eqref{eq:DDD}, we have 
\begin{equation*}
|D(v, v')| \le | D(v, v_0)| + |D(v', v_0) |+|\mathcal A| \le 3 \b_1 n^{r-1}.
\end{equation*}
As $v, v_0$ and $v', v_0$ are both consistent pairs, the fact that $v, v'\in V_1$ implies that $| D'(v, v_0) | \le \b_1 n^{r'-1}$ and $| D'(v', v_0) | \le \b_1 n^{r'-1}$. Thus, by \eqref{eq:DDD2} we have that $|D'(v,v')|\leq 2 \beta _2 n^{r'-1}+|\mathcal A'| \leq
3\beta _1 n^{r'-1}$.

Next consider any distinct $v, v'\in V_2$. By the definition of $V_2$ and \eqref{eq:DDD}, we have 
\[
|D(v, v')| \le | C(v, v_0)| + |C(v', v_0) |+|\mathcal A| \le 3 \b_1 n^{r-1}.
\]
We can show that $|D'(v, v')| \le 3\b_1 n^{r'-1}$ as before.

Consider any $v \in V_1$ and $v'\in V_2$. Suppose that $T \subseteq V-\{v_0,v,v'\}$ is an $(r-1)$-set such that $T \cup \{ v,v'\} \subseteq C(v,v')$. So $T+v$ and $T+v'$
have the same color. Hence, $T+v_0$ has either a different color to $T+v$ or the same color as $T+v'$. Thus,
\begin{equation*}
|C(v, v')|\leq | D(v, v_0)| +| C(v', v_0)|+|\mathcal A| \le 3\b_1 n^{r-1},
\end{equation*}
where the second inequality follows by the definitions of $V_1$ and $V_2$. An analogous argument gives that $|C'(v, v')| \le 3\b_1 n^{r'-1}$.
\endproof

\medskip

Once we have obtained more information we will prove that $|V_1|$ and $|V_2|$ are close to $n/2$. However, to prove Claim~\ref{clm:ai} we first require the following weaker lower bounds on $|V_1|$ and $|V_2|$.
\begin{claim}\label{claim:V12}
$|V_1|, |V_2| \ge \frac{n}{20 r^2}$. 
\end{claim}

\begin{proof}
Suppose that $|V\setminus V_1| =cn$ for some $c$. Let us consider
\[
m := \sum_{vv'\in \binom{V}{2} } | D (v, v')| \le \sum_{vv' \in \binom{V_1}{2} } |D (v, v') | + \sum_{v\not\in V_1, v'\in V} |D (v, v') |.
\]
By Claim~\ref{ab} (a), the first term is at most $\binom{n}{2} 3\beta_1 n^{r-1} < 2\beta_1 n^{r+1}$.
The second term is at most $cn \binom{n}{r-1} n < \frac{ c}{ (r-1)!} n^{r+1}$. Together 
this gives that $m < (2\beta_1 + {c}/ {(r-1)!}) n^{r+1}$. 

On the other hand, the definition of $B_{n,k}$ yields $|X_1| = \lceil \binom{n}{r}/2 \rceil$.
Applying Lemma~\ref{lem:KK} to the $r$-uniform hypergraph $F:=(V, X_1)$, we derive that $m \ge (\frac{1}{4} - o(1))\binom{n}{r+1} \ge \frac{n^{r+1} }{5(r+1)!}$ (note that $D_F(v, v')= D(v, v')$ for any distinct $v$ and $v'$). Together with the upper bound for $m$, it follows that 
\[
2 \beta_1 + \frac{c}{ (r-1)! } > \frac{1}{5(r+1)!} \quad \text{and so} \quad \frac{c}{ (r-1)! } > \frac{1}{ 10(r+1)! },
\] 
since $\beta_1\ll 1/r$. This implies that $c > \frac{1}{10r(r+1)}$. Since $V\setminus V_1 = V_0\cup V_2$ and $|V_0|\le 4\beta_1 n$, we derive that $|V_2| \ge \frac{n}{10r(r+1)} - 4\beta_1 n\ge \frac{n}{20 r^2}$ (again using $\beta_1\ll 1/r$). An analogous argument implies
that $|V_1| \ge \frac{n}{20 r^2}$. 
\end{proof}

\medskip

Given two disjoint subsets $U_1, U_2\subseteq V$ and two integers $i, j\ge 0$,  we call an $(i+j)$-subset $S\subseteq V$ an $U_1^i U_2^j$-set, and write $S\in U_1^i U_2^j$, if $|S\cap U_1| = i$ and $|S\cap U_2|= j$. Let $\a_0 :=  \sqrt{\beta _1}$ and $\a_{i+1} := \a_i + \frac{3}{i+1} (20 r^3)^r \beta_1$ for $0\le i< r$. Set $n_1:= |V_1|$ and $n_2 := |V_2|$.

\begin{claim}\label{clm:ai}
$ \ $
\begin{enumerate}
\item For all $i= 0, \dots, r$,  at least $(1- \a_i) \binom{n_1 - i}{r-i} \binom{n_2}{i}$ $V_1^{r-i} V_2^i$-sets are in $X_{j_i}$, where $j_i \in \{1, 2\}$ and $j_{i+1} \neq j_i$;
\item For all $i= 0, \dots, r'$,  at least $(1- \a_i) \binom{n_1 - i}{r'-i} \binom{n_2}{i}$ $V_1^{r'-i} V_2^i$-sets are in $Y_{j'_i}$, where $j'_i \in \{1, 2\}$ and $j'_{i+1} \neq j'_i$.
\end{enumerate}

\end{claim}
\begin{proof}
We only prove the first assertion as the proof for the second is analogous (even in the
case when $r'=1$). We proceed by induction on $i$. We first apply Lemma~\ref{lem:KK} to the $r$-uniform hypergraph $F$
with vertex set $V_1$ and edge set $X_1 \cap \binom{V_1}{r}$.
If $F$ has edge density $\rho\in [\a_0, 1 - \a_0]$, then there are vertices $v, v'\in V_1$ such that 
  $| D_F(v, v') | \ge \a_0 (1-\a_0) n _1 ^{r-1} / (r+1)!$. However, by Claim~\ref{ab} (a), 
  $| D (v, v') | \le 3 \beta_1 n^{r-1} $. Since $\a_0=\sqrt{\beta _1} \ll 1/r$ and $n_1\geq n/20r^2$ by Claim~\ref{claim:V12}, we have
\[ 
3 \beta_1 n^{r-1} \ge  |D (v, v') | \ge | D_F(v, v') | \ge \frac{\a_0 (1-\a_0)}{ (r+1)! }n_1 ^{r-1}
\ge \frac{\a_0 (1-\a_0)}{ (r+1)! (20r^2)^{r-1}}n ^{r-1} > 3 \beta_1 n^{r-1},
\]
a contradiction. Thus, $\rho > 1- \a_0$ or $\rho < 1 - \a_0$.  If $\rho > 1- \a_0$, then at least $(1- \a_0)\binom{n_1}{r}$ $r$-subsets of $V_1$ are in $X_1$ and we set $j_0:= 1$.  If $\rho < 1 - \a_0$, then at most $\a_0 \binom{n_1}{r}$ $r$-subsets of $V_1$ are in $X_1$. Thus, at least $(1- \a_0)\binom{n_1}{r}$ $r$-subsets of $V_1$ are in $X_2$ and we set $j_0:= 2$.

For the induction step, we actually prove that there are vertices $v_1, \dots , v_r\in V_1$ such that for each $0 \leq i\le r$, at least $(1- \a_i) \binom{n_1 - i}{r-i} \binom{n_2}{i}$ $(V_1\setminus \{v_1, \dots, v_i\})^{r-i} V_2^i$-sets are in $X_j$, where $j= j_0$ if $i$ is even and $j= 3 - i_0$ if $i$ is odd.
Suppose this assertion holds for some $0 \leq i <r$. 
By an averaging argument, there exists a vertex $v_{i+1}\in V_1\setminus \{v_1, \dots, v_i\}$ and at least 
\[
(1- \a_i) \binom{n_1 - i}{r-i} \binom{n_2}{i} \frac{r-i}{n_1 - i} = (1- \a_i) \binom{n_1 - i - 1}{r-i-1} \binom{n_2}{i} 
\]
$(r-1)$-sets $P\in (V_1\setminus \{v_1, \dots, v_{i+1}\})^{r-i-1} V_2^i$ such that $\{v_{i+1}\}\cup P\in X_j$.
This implies that 
\begin{align*}
& \big| \{u\in V_2, P\in (V_1\setminus \{v_1, \dots, v_{i+1}\})^{r-i-1} V_2^i : u\not\in P, \{v_{i+1}\}\cup P\in X_j \} \big| \ge \\
& \qquad (1- \a_i) \binom{n_1 - i - 1}{r-i-1} \binom{n_2}{i} (n_2 - i).
\end{align*}
Since $v_{i+1}\in V_1$, given any $u\in V_2$, we have $| C (v_{i+1}, u)|\le 3\beta_1 n^{r-1}$ by Claim~\ref{ab} (b). Thus,
\begin{align*}
& \big| \{u\in V_2, P\in (V_1\setminus \{v_1, \dots, v_{i+1}\})^{r-i-1} V_2^i : \{u\}\cup P\in X_{3-j} \} \big| \ge \\
& \qquad (1- \a_i) \binom{n_1 - i - 1}{r-i-1} \binom{n_2}{i} (n_2 - i) - 3\beta_1 n^{r-1} n_2.
\end{align*}
Let $m_{i+1}$ be the number of $ (V_1\setminus \{v_1, \dots, v_{i+1}\})^{r-i-1} V_2^{i+1}$-sets in $X_{3-j}$. We thus have 
\begin{align*}
m_{i+1} &\ge (1- \a_i) \binom{n_1 - i - 1}{r-i-1} \binom{n_2}{i} \frac{n_2 - i}{i+1} - \frac{3\beta_1}{i+1} n^{r-1} n_2 \\
& \ge (1- \a_i) \binom{n_1 - i - 1}{r-i-1} \binom{n_2}{i+1} - \frac{3\beta_1}{i+1} n^r.
\end{align*}
It is easy to see that 
$\binom{n_1 - t}{r-t} \binom{n_2}{t} \ge {n_1^{r-t} n_2^t}/{r^r}$ for $0\le t\le r$.
Claim~\ref{claim:V12} states that $n_1, n_2\ge \frac{n}{20r^2}$. Consequently,
\[
\binom{n_1 - t}{r-t} \binom{n_2}{t}\ge \left(\frac{n}{20r^3}\right)^r \quad \text{and so} \quad n^r \le (20 r^3)^r \binom{n_1 - t}{r-t} \binom{n_2}{t},
\]
for all $0\le t\le r$.
Therefore,
\begin{align*}
m_{i+1} & \ge (1- \a_i) \binom{n_1 - i - 1}{r-i-1} \binom{n_2}{i+1} - \frac{3\beta_1}{i+1}  (20 r^3)^r \binom{n_1 - i-1}{r-i-1} \binom{n_2}{i+1} \\
& = (1- \a_{i+1}) \binom{n_1 - i - 1}{r-i-1} \binom{n_2}{i+1},
\end{align*}
as desired.
\end{proof}

Set $\a := \max_{0\le i\le r} \a_i$, $\eta_1 := (2\a)^{1/r} + 4\beta_1$ and  $\tilde{n} := n_1+n_2= |V_1\cup V_2|$.
\begin{claim}\label{clm:n2}
$|V_1|, |V_2| \ge (1- \eta_1) n/2$.
\end{claim}
\begin{proof}
We have that $\tilde{n} \ge (1- 4\beta_1)n$ since $|V_0| \leq 4 \beta _1 n$. 
Let $c:= n_1/ \tilde{n}$. It suffices to show that $\frac12 (1 - (2\a)^{1/r}) \le c\le \frac12 (1 + (2\a)^{1/r}) $ because this implies that 
\[
n_1 = c\tilde{n}\ge \frac12 \big(1 - (2\a)^{1/r} \big) (1- 4\beta_1)n > (1 - \eta_1) \frac{n}{2}
\]
and $n_2= (1-c)\tilde{n} \ge \frac12 (1 - (2\a)^{1/r}) (1- 4\beta_1)n > (1 - \eta_1) {n}/{2}$.

\smallskip
Without loss of generality, assume that $j_0=1$ in the statement of Claim~\ref{clm:ai}. Thus,
\[
|X_1| \ge \sum_{0\le i\le r, \, i \,\rm{even} } (1- \a_i) \binom{n_1 - i}{r-i} \binom{n_2}{i} \ge (1-\a) \sum_{0\le i\le r, \, i \,\rm{even} } \binom{n_1}{r-i} \binom{n_2}{i}  - O(n^{r-1}).
\]
Hence, by Proposition~\ref{evensum}, $|X_1| \ge  (1- \a) \frac{\tilde{n}^r}{2r!} ( 1+ (2c-1)^r ) -O(n^{r-1})$, which implies that  
\[
|X_1| \ge (1-\a) (1- 4\beta_1)^r \frac{n^r}{2r!} \big( 1+ (2c-1)^r \big) - O(n^{r-1}). 
\]
If $(2c-1)^r \ge 2\a$, then 
\[
(1-\a) (1- 4\beta_1)^r ( 1+ (2c-1)^r ) \ge 1 + \a/2, 
\]
since $\beta_1 = \alpha _0 ^2 \leq  \a ^2 \ll 1/r$.
Consequently, $|X_1|\ge (1+ \a/4) \frac{n^r}{2r!}$. This contradicts the assumption $|X_1| = \lceil \binom{n}{r} \rceil \le \frac{n^r}{2r!}$. Thus,
\begin{equation}
\label{eq:c}
 (2c-1)^r < 2\a. 
\end{equation}

If $c\ge 1/2$, then \eqref{eq:c} implies that $c< (1 + (2\a)^{1/r})/2$ and we are done.
Otherwise assume that $c< 1/2$. If $r$ is even, then \eqref{eq:c} implies that $(1-2c)^r = (2c- 1)^r < 2\a$ and 
so $c> (1 - (2\a)^{1/r})/2$. If $r$ is odd, then we apply Claim~\ref{clm:ai} and Proposition~\ref{evensum} obtaining that
\[
|X_2| \ge \sum_{0\le i\le r, \, i \,\rm{odd} } (1- \a_i) \binom{n_1 - i}{r-i} \binom{n_2}{i} \ge (1- \a) \frac{\tilde{n}^r}{2r!} ( 1- (2c-1)^r ) - O(n^{r-1}).
\]
If $1- (2c-1)^r = 1 + (1-2c)^r \ge 1 + 2\a$, then we obtain a contradiction as before  because $|X_2| = \lfloor \binom{n}{r} \rfloor \le \frac{n^r}{2r!}$. Hence, $(1- 2c)^r < 2\a$ and consequently $c> (1 - (2\a)^{1/r})/2$, as required.
\end{proof}

By Claim~\ref{clm:n2}, there exists a partition $ V'_1, V'_2$ of $V$ such that $|V'_1|= \lfloor n/2 \rfloor$, $|V'_2 | = \lceil n/2 \rceil$ and $|V_i\cap V'_i|\ge (1-\eta_1) n/2$ for each $i=1,2$.
\begin{claim}
\label{clm:X1B}
The hypergraph $(V, X_1)$ is $\eta$-close to $\mathcal B_{n,r}(V'_1,V'_2)$ or $\overline{\mathcal B}_{n,r}(V'_1,V'_2)$, and the hypergraph $(V, Y_1)$ is $\eta$-close to $\mathcal B_{n,r'}(V'_1,V'_2)$ or $\overline{\mathcal B}_{n,r'}(V'_1,V'_2)$.
\end{claim}
\begin{proof}
By Claim~\ref{clm:ai}, at least 
\[
\sum_{0\le i\le r, \, i \,\text{even}} (1-\a_i) \binom{n_1-i}{r-i} \binom{n_2}{i} \ge (1-\a) \sum_{0\le i\le r, \, i \,\text{even}}  \binom{n_1}{r-i} \binom{n_2}{i} - O(n^{r-1})
\]
$r$-subsets of $V_1 \cup V_2$ are in $X_{j_0}$ and have an even number of vertices in $V_2$, and at least $(1-\a) \sum_{0\le i\le r, \, i \,\text{odd}}  \binom{n_1}{r-i} \binom{n_2}{i} - O(n^{r-1})$ $r$-subsets of $V_1 \cup V_2$ are in $X_{j_1}$ and have an odd number of vertices in $V_2$. 

Suppose that $r$ is odd. Set $\mathcal{B} := \mathcal{B}_{\tilde{n},r}(V_1, V_2)$. By definition,  
\[
|E(\mathcal{B})| = \sum_{0\le i\le r, \, i \,\text{even}}  \binom{n_1}{r-i} \binom{n_2}{i} 
\quad \text{and} \quad
| E(\overline{\mathcal B} )| = \sum_{0\le i\le r, \, i \,\text{odd}}  \binom{n_1}{r-i} \binom{n_2}{i}.
\] 
Claim~\ref{clm:ai} thus implies that
\begin{equation} \label{eq:EB}
| E(\mathcal B) \cap X_{j_0}| \geq (1- \a) |E(\mathcal{B})| - O(n^{r-1})  
\quad \text{and} \quad
| E(\overline{\mathcal B}) \cap X_{j_1}| \geq (1- \a) | E(\overline{\mathcal B}) |- O(n^{r-1}).
\end{equation}
It follows that 
\[
| E(\mathcal B) \setminus X_{j_0}| = | E(\mathcal B) | - | E(\mathcal B) \cap X_{j_0}| \le \a |E(\mathcal{B})|  + O(n^{r-1}) \le 2\a \binom{n}{r}.
\]
On the other hand, letting $\tilde{X}_{j_0} := X_{j_0} \cap \binom{V\setminus V_0}{r}$, we have
\begin{align*}
|X_{j_0} \setminus E(\mathcal B) | &\le | \tilde{X}_{j_0} \setminus E(\mathcal B)| + |V_0| \binom{n-1}{r-1} \le | \tilde{X}_{j_0} \cap E(\overline{\mathcal B}) | + 4\b_1 n \binom{n-1}{r-1} \\
& = |E(\overline{\mathcal B})| -  |E(\overline{\mathcal B}) \cap X_{j_1}| + 4 \b_1 r \binom{n}{r} \\
&\stackrel{\eqref{eq:EB}}{\le} \a |E(\overline{\mathcal B})| + O(n^{r-1}) +4 \b_1 r \binom{n}{r}  \le 2\a \binom{n}{r}.
\end{align*} 
We thus derive that $|E(\mathcal B) \triangle X_{j_0}|\le 4\a \binom{n}{r}$.

Note that at most $2 \frac{\eta_1}2 n$ vertices of $V$ are not in $(V_1\cap V'_1)\cup (V_2\cap V'_2)$ and each edge in $E(\mathcal B) \triangle E(\mathcal B_{n,r}(V'_1, V'_2))$ must contain such a vertex. Hence, 
\[
| E(\mathcal{B}) \triangle E(\mathcal B_{n,r}(V'_1, V'_2)) | \le  {\eta_1} n \binom{n-1}{r-1}.
\]
Therefore,
\[
| E(\mathcal B_{n,r}(V'_1, V'_2)) \triangle X_{j_0}|  \le {\eta_1} n \binom{n-1}{r-1} + 4\a\binom{{n}}{r}\le {\eta} \binom{n}{r}.
\]
Since $j_0\in \{1, 2\}$, we conclude that $(V, X_1)$ is $\eta$-close to $\mathcal B_{n,r}(V'_1, V'_2)$ or $\overline{\mathcal B}_{n,r}(V'_1, V'_2)$. The case when $r$ is even is analogous. Moreover,  analogous arguments show that $(V, Y_1)$ is $\eta$-close to $\mathcal B_{n,r'}(V'_1, V'_2)$ or $\overline{\mathcal B}_{n,r'}(V'_1, V'_2)$.
\end{proof}

\begin{claim}
$H$ is $\eps$-close to $\mathcal B_{n,k}(V'_1, V'_2)$ or $\overline{\mathcal B}_{n,k}(V'_1, V'_2)$.
\end{claim}

\begin{proof}
Since we always consider the partition $ V'_1, V'_2$ of $V$, we write, for example,
$\mathcal B_{n,k}$ instead of $\mathcal B_{n,k}[V'_1,V'_2]$ throughout the proof.
 We call a set $S\subseteq V$ \emph{even} if $|S\cap V'_1|$ is even. Otherwise we say that $S$ is  \emph{odd}.
Thus, $E(\mathcal{B}_{n,r})$ consists of all odd $r$-sets, and $E(\overline{\mathcal B}_{n,r'})$ consists of all even $r'$-sets. For convenience, we use $X_i, Y_i$ to denote the hypergraphs $(V, X_i), (V, Y_i)$ for $i=1, 2$.

There are four possible cases:
\begin{enumerate}
\item $X_1 = \B_{n, r} \pm \eta n^r$ and $Y_1= \B_{n,r'} \pm \eta n^{r'}$,
\item $X_1 = \B_{n, r} \pm \eta n^r$ and $Y_1= \oB_{n,r'} \pm \eta n^{r'}$,
\item $X_1 = \oB_{n, r} \pm \eta n^r$ and $Y_1= \B_{n,r'} \pm \eta n^{r'}$,
\item $X_1 = \oB_{n, r} \pm \eta n^r$ and $Y_1= \oB_{n,r'} \pm \eta n^{r'}$.
\end{enumerate}
We claim that $H= \B_{n,k} \pm \eps n^k$ under Cases 2 and 3, and $H= \oB_{n,k} \pm \eps n^k$ under Cases 1 and 4. Below we show that $H= \B_{n,k} \pm \eps n^k$ under Case 2. (The other cases are analogous.)

Assume that $X_1 = \B_{n, r} \pm \eta n^r$ and $Y_1= \oB_{n,r'} \pm \eta n^{r'}$.  Our goal is to show that
$|E(H)\triangle E(\mathcal B_{n,k})| \leq \eps n^{k}$.

First we show that $|E(H)\setminus E(\mathcal B_{n,k})| \leq \eps n^k /2$. Consider any $k$-set $Q$
from $ E(H)\setminus E(\mathcal B_{n,k})$. Since $Q \not \in  E(\mathcal B_{n,k})$ (and thus $|Q\cap V'_1|$ is even), $Q$ can be
 partitioned into (i) an even $r$-set $\u{x}$ and an even $r'$-set $\u{y}$ or (ii) an odd $r$-set $\u{x}$ and an odd $r'$-set $\u{y}$. 
As $Q \in E(H)$, in both cases we have that $\{ \u{x} , \u{y} \} \in E(G)$. Thus,
$$ |E(H)\backslash E(\mathcal B_{n,k}) | \leq |\Sigma_1|+ |\Sigma_2|,$$
where $\Sigma _1$ is the set of all disjoint $\u{w}\in X^r$, $\u{z} \in Y^{r'}$ such that $\u{w}$ and $\u{z}$ are even and $\{\u{w} , \u{z} \} \in E(G)$ and $\Sigma _2$ is the set of all disjoint $\u{w}\in X^r$, $\u{z} \in Y^{r'}$ such that $\u{w}$ and $\u{z}$ are odd and $\{\u{w} , \u{z} \} \in E(G)$.

Since $X_1 = \mathcal B_{n,r} \pm \eta n^{r}$ (and thus $X_2 = \oB_{n,r} \pm \eta n^{r}$), there are at most $\eta n^{r} \binom{n}{r'}\leq \eta n^{k}$ pairs $\{ \u{w},\u{z} \} \in \Sigma_1$ such that $\u{w}$ is not in $X_2$. Since $Y_1 = \oB_{n,r'} \pm \eta n^{r'}$, there are at most $\eta n^{r'} \binom{n}{r}\leq \eta n^{k}$ pairs $\{ \u{w},\u{z} \} \in \Sigma_1$ such that $\u{z}$ is not in $Y_1$. By the structure of $G$, at most $\beta \binom{n}{r} \binom{n}{r'} < \beta n^k$ pairs $\u{w}\in X_2$, $\u{z} \in Y_1$ are such that $\{\u{w} , \u{z} \} \in E(G)$. 
Together, this implies that $|\Sigma_1| \leq (\eta +\eta+ \beta)n^{k} \leq \eps n^{k}/4 $. 

Since $X_1 = \mathcal B_{n,r} \pm \eta n^{r}$, there are at most $\eta n^{r} \binom{n}{r'}\leq \eta n^{k}$ pairs $\{ \u{w},\u{z} \} \in \Sigma _2$ such that $\u{w}$ is not in $X_1$. Since $Y_2 = \B_{n,r'} \pm \eta n^{r'}$, there are at most $\eta n^{r'} \binom{n}{r}\leq \eta n^{k}$ pairs $\{ \u{w},\u{z} \} \in \Sigma _2$ such that $\u{z}$ is not in $Y_2$. By the structure of $G$, at most $\beta \binom{n}{r} \binom{n}{r'} < \beta n^k$ pairs $\u{w}\in X_1$, $\u{z} \in Y_2$ are such that $\{\u{w} , \u{z} \} \in E(G)$. 
Together, this implies that $|\Sigma_2| \leq (\eta +\eta+ \beta)n^{k} \leq \eps n^{k}/4 $. 
So indeed,
$|E(H) \setminus E(\mathcal B_{n,k})|\leq \eps n^{k}/2 $.

Next we show that $|E(\mathcal B_{n,k})\setminus E(H)|\leq \eps n^{k} /2$. Consider any $k$-set $Q $ from $E(\mathcal B_{n,k})\setminus E(H)$. Since $Q \in  E(\mathcal B_{n,k})$ (and thus $|Q\cap V'_1|$ is odd), we can partition $Q$ into (i) an even $r$-set $\u{x}$ and an odd $r'$-set $\u{y}$ or (ii) an odd $r$-set $\u{x}$ and an even $r'$-set $\u{y}$. As $Q \not \in E(H)$,
in both cases we have that $\{ \u{x} , \u{y} \} \not\in E({G})$.
Thus,
$$ |E(\mathcal B_{n,k})\setminus E(H)| \leq |\Gamma_1|+|\Gamma_2|,$$
where $\Gamma_1$ is the set of all disjoint $\u{w}\in X^r$, $\u{z} \in Y^{r'}$ such that $\u{w}$ is even,
$\u{z}$ is odd and $\{\u{w} , \u{z} \} \not\in E({G})$ and $\Gamma_2$ is the set of all disjoint $\u{w}\in X^r$, $\u{z} \in Y^{r'}$ such that $\u{w}$ is odd,
$\u{z}$ is even and $\{\u{w} , \u{z} \} \not\in E({G})$.

 Since $X_2 = \oB_{n,r} \pm \eta n^{r}$, there are at most $\eta n^{r} \binom{n}{r'}\leq \eta n^{k}$ pairs $\{ \u{w},\u{z} \} \in \Gamma_1$ such that $\u{w}$ is not in $X_2$. Since $Y_2 = \B_{n,r'} \pm \eta n^{r'}$, there are at most $\eta n^{r'} \binom{n}{r}\leq \eta n^{k}$ pairs $\{ \u{w},\u{z} \} \in \Gamma_1$ such that $\u{z}$ is not in $Y_2$. By the structure of $G$,  at most $\beta \binom{n}{r} \binom{n}{r'} < \beta n^k$ pairs $\u{w}\in X_2$, $\u{z} \in Y_2$ are such that $\{\u{w} , \u{z} \}\not\in E(G)$. 
Together, this implies that $|\Gamma_1| \leq (\eta +\eta+ \beta)n^{k} \leq \eps n^{k}/4 $.

Since $X_1 = \B_{n,r} \pm \eta n^{r}$, there are at most $\eta n^{r} \binom{n}{r'}\leq \eta n^{k}$ pairs $\{ \u{w},\u{z} \} \in \Gamma_2$ such that $\u{w}$ is not in $X_1$. Since $Y_1 = \oB_{n,r'} \pm \eta n^{r'}$, there are at most $\eta n^{r'} \binom{n}{r}\leq \eta n^{k}$ pairs $\{ \u{w},\u{z} \} \in \Gamma_2$ such that $\u{z}$ is not in $Y_1$. By the structure of $G$,  at most $\beta \binom{n}{r} \binom{n}{r'} < \beta n^k$ pairs $\u{w}\in X_1$, $\u{z} \in Y_1$ are such that $\{\u{w} , \u{z} \}\not\in E(G)$. 
Together, this implies that $|\Gamma_2| \leq (\eta +\eta+ \beta)n^{k} \leq \eps n^{k}/4 $. So indeed,
$|E(\mathcal B_{n,k})\setminus E(H)|\leq \eps n^{k} /2$. Therefore $|E(H)\triangle E(\mathcal B_{n,k})| \leq \eps n^{k}$, as desired.

\end{proof}

This thus completes the proof of Lemma~\ref{lem:GH}.

\section*{Acknowledgement}
We thank Oleg Pikhurko for a helpful comment on this problem.

\medskip

{\footnotesize \obeylines \parindent=0pt

\begin{tabular}{lll}

Andrew Treglown                     &\ &  Yi Zhao \\
School of Mathematical Sciences					&\ &  Department of Mathematics and Statistics \\
Queen Mary, University of London  &\ &  Georgia State University \\
Mile End Road                  &\ &  Atlanta \\
London		&\ &  Georgia 30303\\
E1 4NS												&\ &  USA\\
UK											&\ &
\end{tabular}
}

{\footnotesize \parindent=0pt

\it{E-mail addresses}:
\tt{treglown@maths.qmul.ac.uk}, \tt{yzhao6@gsu.edu}}

\end{document}